\documentclass[12pt, noinfoline]{imsart} 

\usepackage[utf8]{inputenc} 

\usepackage{geometry} 
\usepackage{amsthm, amsfonts, amssymb, amsmath,enumitem, natbib, bbm}
\usepackage{mathtools}
\mathtoolsset{showonlyrefs}
\newtheorem{thm}{Theorem}[section]
\newtheorem{lem}{Lemma}[section]
\newtheorem{cor}{Corollary}[section]
\newtheorem{prop}{Proposition}[section]
\newtheorem{exam}{Example}[section]
\newcounter{myalgctr}
\newenvironment{rem}{
   \vskip1mm\indent
   \refstepcounter{myalgctr}
   \textbf{Remark \themyalgctr}
   }{\hfill$\diamond$\par}  
\numberwithin{myalgctr}{section}
\providecommand{\norm}[1]{\left\lVert#1\right\rVert}
\DeclareMathOperator*{\argmin}{arg\,min}

\usepackage{graphicx} 
\usepackage{color, float}

\usepackage{booktabs} 
\usepackage{array} 
\usepackage{verbatim} 
\usepackage{subfig} 
\usepackage[colorlinks=true, a4paper=true, pdfstartview=FitV,
linkcolor=blue, citecolor=blue, urlcolor=blue]{hyperref}
\makeatletter
\def\namedlabel#1#2{\begingroup
    #2%
    \def\@currentlabel{#2}%
    \phantomsection\label{#1}\endgroup
}
\makeatother
\newcommand{\vertiii}[1]{{\left\vert\kern-0.25ex\left\vert\kern-0.25ex\left\vert #1
    \right\vert\kern-0.25ex\right\vert\kern-0.25ex\right\vert}}

\usepackage[nottoc,notlof,notlot]{tocbibind} 
\usepackage[titles,subfigure]{tocloft} 

\begin{document}
 \begin{frontmatter}
\title{Deterministic Inequalities for Smooth M-estimators}
\runtitle{Deterministic Inequalities for $M$-estimators}
\begin{aug}
  \author{\fnms{Arun Kumar}  \snm{Kuchibhotla}\ead[label=e1]{arunku@wharton.upenn.edu}}

  \runauthor{Arun K. Kuchibhotla}

  \affiliation{University of Pennsylvania}

  \address{University of Pennsylvania\\ \printead{e1}}

\end{aug}
\begin{abstract}
Ever since the proof of asymptotic normality of maximum likelihood estimator by \cite{Cram46}, it has been understood that a basic technique of the Taylor series expansion suffices for asymptotics of $M$-estimators with smooth/differentiable loss function. Although the Taylor series expansion is a purely deterministic tool, the realization that the asymptotic normality results can also be made deterministic (and so finite sample) received far less attention. With the advent of big data and high-dimensional statistics, the need for finite sample results has increased. In this paper, we use the (well-known) Banach fixed point theorem to derive various deterministic inequalities that lead to the classical results when studied under randomness. In addition, we provide applications of these deterministic inequalities for crossvalidation/subsampling, marginal screening and uniform-in-submodel results that are very useful for post-selection inference and in the study of post-regularization estimators. Our results apply to many classical estimators, in particular, generalized linear models, non-linear regression and cox proportional hazards model. Extensions to non-smooth and constrained problems are also discussed.
\end{abstract}
\end{frontmatter}
\section{Introduction}
One of the basic problems of statistics concerns estimation of parameters or functionals of a population based on a sample of observations. A large class of estimators in statistics are obtained as minimizers of some function of the observations, that is, estimators $\hat{\theta}_n$ are obtained as
\begin{equation}\label{eq:FirstMEstimatorDefinition}
\hat{\theta}_n := \argmin_{\theta\in\Theta_n}\,\mathbb{M}_n(\theta; Z_1, \ldots, Z_n),
\end{equation}
for some parameter space $\Theta_n$ (possibly depending on the sample size $n$) and a function $\mathbb{M}_n(\theta; Z_1, \ldots, Z_n)$ written explicitly as a function of the parameter argument $\theta$ and observations $Z_1, \ldots, Z_n$. Here the random variables $Z_1, \ldots, Z_n$ take values in a measurable space (not necessarily Euclidean) and are not required to be either independent or to be identically distributed. For instance, the ordinary least squares (OLS) linear regression estimator $\hat{\beta}_n$ based on regression data $Z_1 = (X_1, Y_1),$ $\ldots$, $Z_n = (X_n, Y_n)\in\mathbb{R}^{p + 1}$ is given by
\begin{equation}\label{eq:OLSEstimator}
\hat{\beta}_n := \argmin_{\theta\in\mathbb{R}^p}\,\frac{1}{n}\sum_{i=1}^n \left\{Y_i - X_i^{\top}\theta\right\}^2.
\end{equation}
Estimators of the type~\eqref{eq:FirstMEstimatorDefinition} are referred to as $M$-estimators in \cite{VdvW96}. 

The study of the (asymptotic) properties of $M$-estimators is an ever-evolving research area in statistics. One of the most general and widely used frameworks for this study can be found in Section 3.2 of \cite{VdvW96}. The results in this section require a stochastic equicontinuity assumption which in turn requires controlling the supremum of a stochastic process. Under certain differentiability assumptions on $\mathbb{M}_n$, this equicontinuity assumption can be easily validated. Although there are numerous empirical process techniques to verify the equicontinuity assumption under independence, we do not know of such general techniques in case of dependent observations. In this paper, we provide a general way of proving deterministic inequalities for understanding the estimator $\hat{\theta}_n$ which only requires randomness in verifying convergence of remainder terms to zero. This approach sacrifices the level of generality of \cite{VdvW96} with certain smoothness assumptions but provides optimal rates as well as tail bounds.

Deterministic inequalities for $M$-estimators to be discussed were shown in \cite{Uniform:Kuch18} for the OLS estimator~\eqref{eq:OLSEstimator}. The proofs there are much easier because of the explicit/closed-form OLS solution. It should be mentioned here that the idea of deriving results for $M$-estimators without assuming a particular dependence structure is not new and can be found in the works of \cite{VdvW96}, \cite{Yuan98}, \cite{HjortPollard}, \cite{Geyer13} and \cite{MR3683985}. This list is by no means exhaustive.
\subsection{The need for Deterministic Inequalities}
A starting point for this paper is \cite{HjortPollard}; this paper was available since 1993. One of the main conclusions of \cite{HjortPollard} is the following: pointwise convergence of a random convex objective function to a fixed function implies the consistency, rate of convergence and asymptotic distribution of the (global) minimizer. Some natural follow-up questions are ``what happens if the parameter space changes with $n$? What if the dimension grows? What if the dependence between the observations changes with $n$?''. These questions are also hard to answer from classical asymptotic normality results. However, the study of remainders in a single deterministic inequality can answer all these questions in a simple way. In this respect, deterministic inequalities unify the study of properties of the solutions of an estimating equation. 

Another source of motivation for deterministic inequalities is the need for understanding a collection of many estimators in some statistical applications. For instance, model-selection plays a pivotal role in data analysis where it is of interest to choose a ``good parsimonious'' model out of a (fixed) collection of models. In this case, to understand how a  model-selection procedure works, it is necessary to study \emph{simultaneously} the properties of all the  estimators in the collection of models. Depending on the number of models in the collection and the dependence between the observations, classical asymptotic results do not provide a clear understanding while deterministic inequalities (if they exist) provide detailed knowledge without any difficulty. Some of these examples will be presented later.
\subsection{Can we expect Deterministic Inequalities?}
As hinted in the abstract, classical asymptotic normality results are essentially based on the Taylor series expansion which is a deterministic tool. So, a first order Taylor series expansion of the estimating function (with explicit bounds on the remainder) and inversion implies explicit bounds for the solutions of the estimating equation.

A clearer picture can be seen through functionals. Let $\hat{\theta}_n := \theta(P_n)$ and $\theta_0 := \theta(P)$ be defined, respectively, as solutions of the equations
\[
\int \psi(\theta; w)dP_n(w) = 0\quad\mbox{and}\quad\int \psi(\theta; w)dP(w) = 0,
\] 
where $\psi(\cdot; \cdot)$ is a fixed function and $P_n$ is the empirical probability measure based on the observations $W_1$, $W_2$, $\ldots$, $W_n$. The functional here is $\theta(\cdot)$ which is a function on the space of all probability measures. Under most independence and dependence settings, it is known that $P_n$ is ``close'' to $P$ (in a suitable metric). So, if the functional $\theta(\cdot)$ is continuous, then we get
\[
\theta(P_n) - \theta(P)~\approx~ 0.
\]
Further if the functional $\theta(\cdot)$ is Fr{\'e}chet differentiable, then
\[
\theta(P_n) - \theta(P) - \theta'(P;\,P_n - P) = o(d(P_n, P)),
\]
for some metric $d(\cdot, \cdot)$. Here $\theta'(P; P_n - P)$ denotes the directional derivative of $\theta(\cdot)$ at $P$ in the direction of $P_n - P$. This is essentially the framework of von Mises calculus; see~\cite{clarke2018robustness} for details. It is clear that if the continuity assumption is replaced by H{\"o}lder continuity, that is, $\norm{\theta(P_n) - \theta(P)} \le d^{\alpha}(P_n, P)$ for some metric $d(\cdot, \cdot)$ and $\alpha > 0$, then we got a deterministic inequality for estimation error. Similarly, if the differentiability assumption is made precise in terms of some explicit bounds, then we can write
\[
\norm{\theta(P_n) - \theta(P) - \theta'(P;\,P_n - P)} \le Cd^{1 + \alpha}(P_n, P),
\]
for some constant $C > 0$ and $\alpha > 0$. This again provides a deterministic inequality for an expansion of the estimator. In the following sections, we provide these explicit bounds for various $M$-estimation problems. 

There is a also a rich literature in the field of mathematical programming where the problem
\[
\int \psi(\theta; w)dP_n(w) = 0,
\]
is seen as a ``perturbation'' of the problem $\int \psi(\theta; w)dP(w) = 0$ (since $P_n\approx P$). There are a large number of sensitivity and stability results that show how much different $\theta(P_n)$ is from $\theta(P)$; see, for example, \cite{romisch2007stability}, \cite{rockafellar2009variational}. In the current paper, we restrict mostly to smooth $M$-estimators, meaning differentiable $\psi(\cdot, w)$ with a H{\"o}lder continuous derivative. There are numerous results in mathematical programming literature that work for non-smooth functions and also, $M$-estimators with constraints. We hope to tackle these additional problems in the future.   
\subsection{Organization}
The remainder of the paper is organized as follows. In Section~\ref{sec:Banach}, we state a Banach fixed point theorem in a form suitable for our purposes. This result appeared in a similar form in \cite{Yuan98} and \cite{jacod2018review}. Also, a result similar to Newton-Kantorovich theorem is provided. All our subsequent results follow from these results. Applications to $M$-estimators based on convex loss functions including generalized linear models (GLMs) and Cox proportional hazards model are given in Sections~\ref{sec:GLM} and~\ref{sec:Cox}, respectively. In Section~\ref{sec:NonConvex}, we provide deterministic inequalities for least squares non-linear regression. In Section~\ref{sec:Constrained}, we provide a deterministic inequality for an equality constrained minimization problem. In Section~\ref{sec:Applications}, we present three applications of our deterministic inequalities for cross-validation/subsampling, marginal screening and post-selection inference. We conclude with some remarks and future directions in Section~\ref{sec:Conclude}.
\section{The Basic Result}\label{sec:Banach}
In this section, we state a basic inversion theorem that implies existence of solutions to an equation in a certain neighborhood. This result is of primary importance to us since it gives explicit bounds on the radius of the neighborhood and so can provide finite sample results in statistical applications. The following result is stated in a similar form in \cite{Yuan98}. Define for any $\theta_0\in\mathbb{R}^q$ and $r > 0$, the closed ball as
\[
B(\theta_0, r) := \left\{\theta\in\mathbb{R}^q:\,\norm{\theta - \theta_0}_2 \le r\right\}.
\]
For any matrix $A$, let $\norm{A}_{op}$ denote the operator norm of $A$. Also, for any function $f(\cdot)$, $\nabla f(\cdot)$ and $\nabla_2f(\cdot)$ denote the gradient and Hessian of $f(\cdot)$.
\begin{thm}\label{thm:NonSingular}
Let $f(\cdot)$ be an everywhere differentiable mapping from an open subset of $\mathbb{R}^q$ into $\mathbb{R}^q$. Let $A$ be a non-singular matrix. If for some $\theta_0\in\mathbb{R}^{q}$, $r > 0$ and $\varepsilon \in [0, 1]$,
\begin{equation}\label{assump:Contraction}
\norm{A^{-1}\left(A - \nabla f(\theta)\right)}_{op} \le \varepsilon\quad\mbox{for all}\quad \theta\in B(\theta_0, r),
\end{equation}
and 
\begin{equation}\label{assump:FixedPoint}
\norm{A^{-1}f(\theta_0)}_2 \le r(1 - \varepsilon),
\end{equation}
then there exists a unique vector $\theta^*\in B(\theta_0, r)$ satisfying $f(\theta^*) = 0$ and
\[
\frac{1}{1 + \varepsilon}\norm{A^{-1}f(\theta_0)}_2 \le \norm{\theta^{\star} - \theta_0}_2 \le \frac{1}{1 - \varepsilon}\norm{A^{-1}f(\theta_0)}_2.
\]
\end{thm}
The proof of this theorem is given in Appendix~\ref{appsec:NonSingular}. Also, see~\cite{jacod2018review} for related results and applications of this result for the asymptotics of statistical estimating functions.

Interestingly (from the statistical viewpoint), Theorem~\ref{thm:NonSingular} does not require any special properties for $\theta_0$. This particular fact becomes important in applications related to subsampling/cross-validation in Section~\ref{subsec:CrossValidation}. In the following sections, we apply this result with $f(\theta) = \nabla \mathbb{M}_n(\theta), A = \nabla_2\mathbb{M}_n(\theta_0)$ (which requires the objective function $\mathbb{M}_n(\cdot)$ to be twice differentiable) and deterministically, the result shows that the estimation error $\|\hat{\theta} - \theta_0\|_2$ is up to a constant factor same as $\norm{\nabla \mathbb{M}_n(\theta_0)}_2$. The classical results mostly use the choice $\theta_0$ that satisfies $\mathbb{E}\left[\nabla\mathbb{M}_n(\theta_0)\right] = 0$. If the objective function $\mathbb{M}_n(\cdot)$ is an average, then the control of $\norm{\nabla \mathbb{M}_n(\theta_0)}_2$ is the same as controlling a mean zero average which is studied in probability and statistics for almost all practically useful dependence structures. 

For simplicity, the result is stated with $\mathbb{R}^q$ as the domain but it is not difficult to extend the proof to more general Banach spaces under Fr{\'e}chet differentiability. This extension is useful in deriving deterministic inequalities for smoothing spline estimators for non-parametric regression/density estimation; see~\cite{shang2010convergence} and~\cite{shang2013local}. Another interesting extension can be obtained by replacing Euclidean norm-$\norm{\cdot}_2$ by a general norm $\norm{\cdot}_N$ on $\mathbb{R}^q$. In this case, the assumptions~\eqref{assump:Contraction} and~\eqref{assump:FixedPoint} need to be rewritten as
\[
\norm{A^{-1}(A - \nabla f(\theta))}_{N\to N} \le \varepsilon,\quad\mbox{for all}\quad \theta\in B_{r,N}(\theta_0),
\]
and
\[
\norm{A^{-1}f(\theta_0)}_N \le r(1 - \varepsilon),
\]
where $\norm{\cdot}_N$ represents a norm on $\mathbb{R}^q$ and for any matrix $K$,
\[
\norm{K}_{N\to N} := \sup\{x^{\top}Kx:\,\norm{x}_N \le 1\},
\]
and $B_{r,N}(\theta_0) := \{\theta\in\mathbb{R}^q:\,\norm{\theta - \theta_0}_N \le r\}$.

In many statistical applications of interest, it is also of interest to prove a first order (influence function) approximation for the estimator. For these applications, we present the following extended result under a strengthening of the assumptions of Theorem~\ref{thm:NonSingular}. The following theorem is closely related to the well-known Newton-Kantorovich Theorem about Newton's method of root finding.
\begin{thm}\label{thm:NewtonStepDifferentiable}
Let $f(\cdot)$ be an everywhere differentiable mapping from an open subset of $\mathbb{R}^q$ into $\mathbb{R}^q$. If for some $\theta_0\in\mathbb{R}^{q}$, $L\ge 0$, and $\alpha\in(0, 1]$,
\begin{equation}\label{eq:HolderContinuity}
\norm{[\nabla f(\theta_0)]^{-1}\left(\nabla f(\theta_0) - \nabla f(\theta)\right)}_{op} \le L\norm{\theta - \theta_0}_2^{\alpha},
\end{equation}
whenever $(3L)^{1/\alpha}\norm{\theta - \theta_0}_2 \le 1$ and 
\begin{equation}\label{eq:AssumptionNewton}
\norm{[\nabla f(\theta_0)]^{-1}f(\theta_0)}_2 \le \frac{2}{3(3L)^{1/\alpha}},
\end{equation}
then there exists a unique solution $\theta^{\star}$ of $f(\theta) = 0$ in $B(\theta_0, 1.5\|[\nabla f(\theta_0)]^{-1}f(\theta_0)\|_2)$ and that solution $\theta^{\star}$ satisfies
\begin{equation}\label{eq:ExpansionNewton}
\norm{\theta^{\star} - \theta_0 + [\nabla f(\theta_0)]^{-1}f(\theta_0)}_2 \le (1.5)^{1 + \alpha}L\norm{[\nabla f(\theta_0)]^{-1}f(\theta_0)}_2^{1 + \alpha}.
\end{equation}
\end{thm}

The expansion~\eqref{eq:ExpansionNewton} of Theorem~\ref{thm:NewtonStepDifferentiable} is essentially proving that the first iteration of Newton's scheme is ``close'' to the true solution $\theta^{\star}$ and is a very special case of the general superlinear convergence statement\footnote{A sequence of iterates $\{\theta_n\}_{n\ge 0}$ is said to converge superlinearly if $\norm{\theta_{n+1} - \theta_n}_2 = o(\norm{\theta_n - \theta_{n-1}}_2)$ as $n\to\infty.$ In our case $\theta_0$ is the initial point and $\theta_1 = \theta_0 - (\nabla f(\theta_0))^{-1}f(\theta_0)$ is the first iterate.}. The classical Newton-Kantorovich theorem (Theorem 1.1 of~\cite{yamamoto1985unified}) usually requires a slightly stronger condition on the derivative of $f$ with $\alpha = 1$ and proves explicit bounds for all iterations of the Newton's scheme. Also, see~\cite{clarke2007convergence} for some applications in statistical problems.

An important message from Theorem~\ref{thm:NewtonStepDifferentiable} is that any iterative algorithm that requires conditions only on the initial point\footnote{Convergence analysis that require conditions only on the initial point is usually referred to as semilocal analysis.} ($\theta_0$) and proves superlinear convergence can be used to prove expansion results like~\eqref{eq:ExpansionNewton}. 
Apart from the classical Newton's method (that requires differentiable $f$), there are numerous extensions allowing for non-smooth $f$ including B-differentiable functions (\cite{qi1993nonsmooth}) and normal mappings (\cite{robinson1994newton}). For a general treatment of Newton's method, see~\cite{argyros2008convergence}.

From the proof, it is easy to replace the right hand side of assumption~\eqref{eq:HolderContinuity} by any non-decreasing function $\omega(\cdot)$ of $\norm{\theta - \theta_0}_2$; this extension is useful for nonlinear regression as in Section~\ref{sec:NonConvex}. As before, an extension of Theorem~\ref{thm:NewtonStepDifferentiable} to Banach spaces is possible. In fact, Theorem 1.1 of~\cite{yamamoto1985unified} holds for Banach spaces. 

Most commonly used $M$-estimators in statistics or machine learning are based on objective functions that are averages. In this case, $f(\cdot)$ and $\nabla f(\cdot)$ are also averages. Averages (under independence as well as dependence) have been the subject of investigation for decades in statistics and probability literature. Thus, our results imply that randomness plays the role only in controlling the averages appearing as the remainders.

\begin{rem}\,(Implications for the Landscape of Non-convex Losses)\label{rem:Landscape} 
Deterministic inequalities of the type obtained in Theorems~\ref{thm:NonSingular} and~\ref{thm:NewtonStepDifferentiable} have implications for local minimizers in statistical applications. Suppose there are $n$ identically distributed observations $W_1, \ldots, W_n$ and the parameter of interest is $\theta_0\in\mathbb{R}^p$ that is defined as a global minimizer of $\mathbb{E}[\ell(\theta, W_1)]$. Then a natural estimator of $\theta_0$ is
\[
\hat{\theta}_n := \argmin_{\theta\in\mathbb{R}^{p}}\,\frac{1}{n}\sum_{i=1}^n \ell(\theta, W_i).
\]
If the loss function $\ell(\cdot, w)$ is a non-convex function, then it is in general very hard to obtain a global minimizer $\hat{\theta}_n$. For this reason, it is of significant interest to understand the behavior of local minimizers or critical points of the sample loss function. Recently \cite{mei2016landscape} proved that the landscape of the sample loss function is similar to that the population loss function under certain assumptions including independent and identically distributed (iid) observations. 

Using a deterministic inequality, this fact becomes clear. Suppose, for some $K\ge 1$, $\theta_0^{(1)}, \theta_0^{(2)}, \ldots, \theta_0^{(K)}$ represent the critical values of $\mathbb{E}[\ell(\theta, W_1)]$, that is,
\[
\nabla \mathbb{E}[\ell(\theta, W_1)]\big|_{\theta = \theta_0^{(j)}} = 0,\quad\mbox{for all}\quad 1\le j\le K.
\]
Then under various dependence settings, it is expected that for any $1\le j\le K$,
\[
\norm{\frac{1}{n}\sum_{i=1}^n \nabla\ell(\theta_0^{(j)}, W_i)}_2 = o_p(1).
\]
This implies that assumption~\eqref{eq:AssumptionNewton} of Theorem~\ref{thm:NewtonStepDifferentiable} is satisfied in probability. Thus, by Theorem~\ref{thm:NewtonStepDifferentiable}, it follows that there is a locally unique solution $\hat{\theta}_n^{(j)}$ near $\theta_0^{(j)}$ that furthermore satisfies a linear expansion. This proves that the landscape of the sample loss function is similar to the landscape of the population loss function under more general setting than in~\cite{mei2016landscape}. Note however that our result does not imply critical points for $\mathbb{E}[\ell(\theta, W_1)]$ near the critical points of $\sum_{i=1}^n \ell(\theta, W_i)$.
\end{rem}


\section{Deterministic Inequality for Smooth Convex Loss Functions}\label{sec:GLM}
In this section, we consider $M$-estimators obtained from objective functions that are averages of convex loss functions. Consider the estimator
\begin{equation}\label{eq:ConvexMEst}
\hat{\theta}_n := \argmin_{\theta\in\mathbb{R}^{q}}\,\frac{1}{n}\sum_{i=1}^n L(\theta; W_i),
\end{equation}
for some observations $W_1, \ldots, W_n$ and some loss function $L(\cdot; w)$ that is convex and twice differentiable. Several important examples are as follows:
\begin{exam}[Maximum Likelihood]

\emph{Maximum likelihood estimator (MLE) is one of the most popular estimators in statistics; widely used in practice and backed by the asymptotic efficiency theory. Suppose $W_1, \ldots, W_n$ are iid random variables from a parametric family (of densities) $\{f_{\theta}(\cdot):\,\theta\in\Theta\}$. The MLE is defined as}
\[
\hat{\theta}_n := \argmin_{\theta\in\Theta}\,\frac{1}{n}\sum_{i=1}^n \left\{-\log f_{\theta}(W_i)\right\}.
\]

\emph{If the parametric model family is an exponential family, then the negative log likelihood ($-\log f_{\theta}(\cdot)$) is convex in $\theta$. Even though the construction of the estimator is motivated by the hypothesis/assumption of iid random variables with density belonging to the parametric model postulated, it is important to understand the implication of  misspecification of different directions; see \cite{Hub67} and \cite{Buja14} for a discussion.}
\end{exam}

\begin{exam}[Generalized Linear Models and variants]

\emph{Regression analysis provides a large class of estimation problems which emphasize the problem of estimating the ``relation'' between a response $(Y)$ and a collection of predictors $(X)$. Generalized linear models (GLMs) form an important sub-class of regression models. More generally, we can consider the estimator}
\[
\hat{\theta}_n := \argmin_{\theta\in\mathbb{R}^q}\,\frac{1}{n}\sum_{i=1}^n L(\theta^{\top}X_i; X_i, Y_i),
\]
\emph{for regression data} $W_1 = (X_1, Y_1), \ldots, W_n = (X_n, Y_n).$
\emph{Some specific examples of $L(\cdot; \cdot, \cdot)$ are as follows:}
\begin{enumerate}
	\item \emph{Canonical GLMs are obtained by taking}
	\[
	L(u; x, y) = \psi(u) - yu,
	\]
	\emph{for some convex function $\psi(\cdot)$. For instance, OLS is obtained when $\psi(u) = u^2/2$, logistic regression is obtained when $\psi(u) = \log(1 + \exp(u))$ and Poisson regression is obtained when $\psi(u)= \exp(u)$.}

	\emph{Even though canonical GLMs are motivated from an exponential family for the conditional distribution of $Y$ given $X$, one can consider functions $L(\cdot; \cdot, \cdot)$ that do not correspond to the log-likelihood of an exponential family. For example, probit regression is obtained when}
	\[
	L(u; x, y) = -y\log\Phi(u) - (1 - y)\log(1 - \Phi(u))\quad\mbox{for}\quad y\in[0, 1],
	\]
	\emph{and negative binomial regression corresponds to}
	\[
	L(u; x, y) = -yu + \left(y + \alpha^{-1}\right)\log(1 + \alpha\exp(u)),
	\]
  \emph{for some $\alpha > 0$.}
	\item \emph{Robust regression is an important aspect of practical data analysis and a simple way to robustify an estimator is by ignoring observations that are outliers. In this respect, the loss functions of the form}
	\[
	L(u, x, y) = h(x, y)\ell(u, y),
	\]
	\emph{are of interest. For instance, one can take $\ell(u, y) = \psi(u) - yu$ as in the GLM loss function and take $h(\cdot, \cdot)$ to be a ``down-weighting'' function. Choices of weight functions for robust regression can be found in \cite{loh2017statistical}.}
\end{enumerate}
\end{exam}

Motivated by these examples, we prove the following result for $\hat{\theta}_n$ obtained from the $M$-estimation problem~\eqref{eq:ConvexMEst}. For this result, consider the following notation and assumptions. For the loss function $L(\theta; w)$, let $\nabla L(\theta; w)$ and $\nabla_2 L(\theta; w)$ denote the gradient and the Hessian of the function $L(\theta; w)$ with respect to $\theta$. Set, for any $\theta\in\mathbb{R}^q$, $\delta_{n}(\theta) := 1.5\|[\hat{\mathcal{Q}}_n(\theta)]^{-1}\hat{\mathcal{Z}}_n(\theta)\|_2$, where
\begin{align*}
\hat{\mathcal{Z}}_n(\theta) &:= \frac{1}{n}\sum_{i=1}^n \nabla L(\theta; W_i)\in\mathbb{R}^q\quad\mbox{and}\quad
\hat{\mathcal{Q}}_n(\theta) := \frac{1}{n}\sum_{i=1}^n \nabla_2 L(\theta; W_i)\in\mathbb{R}^{q\times q}.
\end{align*}
Also, define for $u\ge 0$,
\[
C(u, w) := \sup_{\norm{\theta_1 - \theta_2} \le u}\,\sup_{e\in\mathbb{R}^q:\,\norm{e}_2 = 1}\frac{e^{\top}\nabla_2 L(\theta_1, w)e}{e^{\top}\nabla_2 L(\theta_2, w)e}.
\]
Note that if $L(\cdot;w)$ is strictly convex and twice differentiable for each $w$, then $C(u, w)$ is well-defined and positive. Also, note that $C(u, w) \ge 1$ for all $u$ and $w$.
\begin{description}
\item[\namedlabel{eq:ConvexityL}{(A1)}] The function $L(\theta; w)$ is convex and twice differentiable in $\theta$ for every $w$. 
\item[\namedlabel{eq:Event}{(A2)}] Fix any target vector $\theta_0\in\mathbb{R}^q$. The event $\mathcal{E}_n$ occurs where
\[
\mathcal{E}_{n} := \left\{\max_{1\le i\le n}\,C(\delta_n(\theta_0),\, W_i) \le \frac{4}{3}\right\}.
\]
\end{description}
\begin{thm}\label{thm:ConvexMEst}
Under assumptions~\ref{eq:ConvexityL} and~\ref{eq:Event}, there exists a vector $\hat{\theta}_n\in\mathbb{R}^q$ such that
\begin{equation}\label{eq:Part1ConvexMEst}
\hat{\mathcal{Z}}_n(\hat{\theta}_n) = 0,\quad\mbox{and}\quad \frac{1}{2}\delta_{n}(\theta_0) \le \|{\hat{\theta}_n - \theta_0}\|_2 \le \delta_{n}(\theta_0).
\end{equation}
Moreover,
\begin{equation}\label{eq:Part2ConvexMEst}
\norm{\hat{\theta}_n - \theta_0 + [\hat{\mathcal{Q}}_n(\theta_0)]^{-1}\hat{\mathcal{Z}}_n(\theta_0)}_2 \le \max_{1\le i\le n}\left\{C(\delta_{n}(\theta_0), W_i) - 1\right\}\delta_n(\theta_0).
\end{equation}
\end{thm}
\begin{proof}
See Appendix~\ref{appsec:GLM} for a proof.
\end{proof}
\begin{rem}\,(Discussion on the Assumptions)
It is easy to see that assumption~\ref{eq:Event} implies assumption~\ref{eq:ConvexityL} since otherwise the event $\mathcal{E}_n$ cannot hold. Also, from the proof of Theorem~\ref{thm:ConvexMEst}, it follows that the definition of $C(\cdot, \cdot)$ can be replaced by
\[
C(u, w) := \sup_{\norm{\theta - \theta_0} \le u}\,\sup_{e\in\mathbb{R}^q:\,\norm{e}_2 = 1}\max\left\{\frac{e^{\top}\nabla_2 L(\theta, w)e}{e^{\top}\nabla_2 L(\theta_0, w)e}, \frac{e^{\top}\nabla_2 L(\theta_0, w)e}{e^{\top}\nabla_2 L(\theta, w)e}\right\}.
\]
The only difference is that we restrict to $\theta$-vectors that are close to the target $\theta_0$. 

The main reason behind the deterministic inequality is that the objective function can be both upper and lower bounded by (different) quadratic functions. Similar results under additive (rather than a ratio-type) assumption can be found in~\citet[Corollary 3.4]{spokoiny2012parametric}.
\end{rem}
\begin{rem}\,(Linear Representation)
The quantity $C(u, w)$ relates to continuity of the function $\nabla_2L(\cdot, w)$ and usually converges to 1 as $u\to 0$. If this convergence holds, then from Theorem~\ref{thm:ConvexMEst} it follows that as long as $\delta_n(\theta_0) \to 0$,
\[
\hat{\theta}_n - \theta_0~\approx~ [\hat{\mathcal{Q}}_n(\theta_0)]^{-1}\hat{\mathcal{Z}}_n(\theta_0).
\]
The classical proof of asymptotic normality of estimator $\hat{\theta}_n$ obtains an average on the right hand side and the above quantity is not an average because of $\hat{\mathcal{Q}}_n(\theta_0)$. It is easy to replace the average $\hat{\mathcal{Q}}_n(\theta_0)$ by its expectation as follows. Note that
\begin{align*}
&\norm{[\hat{\mathcal{Q}}_n(\theta_0)]^{-1}\hat{\mathcal{Z}}_n(\theta_0) - [{\mathcal{Q}}_n(\theta_0)]^{-1}\hat{\mathcal{Z}}_n(\theta_0)}_2\\
&\qquad\le \|\left[{\mathcal{Q}}_n(\theta_0)\right]^{-1}\hat{\mathcal{Q}}_n(\theta_0) - I\|_{op}\norm{[\hat{\mathcal{Q}}_n(\theta_0)]^{-1}\hat{\mathcal{Z}}_n(\theta_0)}_2\\
&\qquad\le {\|\left[{\mathcal{Q}}_n(\theta_0)\right]^{-1}\hat{\mathcal{Q}}_n(\theta_0) - I\|_{op}\delta_n(\theta_0)}.
\end{align*}
Therefore,
\begin{align}
&\norm{\hat{\theta}_n - \theta_0 + \left[\mathcal{Q}_n(\theta_0)\right]^{-1}\hat{\mathcal{Z}}_n(\theta_0)}_2\nonumber\\
&\qquad \le \left[\max_{1\le i\le n}\,C(\delta_n(\theta_0), W_i) - 1 + \|\left[{\mathcal{Q}}_n(\theta_0)\right]^{-1}\hat{\mathcal{Q}}_n(\theta_0) - I\|_{op}\right]\delta_n(\theta_0).\label{eq:Second2Bound}
\end{align}
In the steps above, it is irrelevant what $\mathcal{Q}_n(\theta_0)$ is but a classical choice is given by
\[
\mathcal{Q}_n(\theta_0) := \frac{1}{n}\sum_{i=1}^n \mathbb{E}\left[\nabla_2L(\theta; W_i)\right].
\]
Finally, if the coefficient of $\delta_n(\theta_0)$ in~\eqref{eq:Second2Bound} is $o_p(1)$, then inequality~\eqref{eq:Second2Bound} proves that
\[
\|\hat{\theta}_n - \theta_0\|_2 = \left(1 + o_p(1)\right)\norm{[\mathcal{Q}_n(\theta_0)]^{-1}\hat{\mathcal{Z}}_n(\theta_0)}_2 = (1 + o_p(1))\frac{2\delta_n(\theta_0)}{3}.
\]
\end{rem}

An application of Theorem~\ref{thm:ConvexMEst} for asymptotic normality of $M$-estimators under a specific dependence structure can be completed using the steps below.
\begin{enumerate}
	\item Define the target $\theta_0$ as a solution to the equation $\mathcal{Z}_n(\theta) = 0$, where
	\[
	\mathcal{Z}_n(\theta) := \frac{1}{n}\sum_{i=1}^n \mathbb{E}\left[\nabla L(\theta; W_i)\right]\in\mathbb{R}^q.
	\]
	This choice of $\theta_0$ ensures that $\mathbb{E}[\hat{\mathcal{Z}}_n(\theta_0)] = 0$ and so, $\hat{\mathcal{Z}}_n(\theta_0)$ becomes a mean zero average. 
	\item Prove that $\|\hat{\mathcal{Z}}_n(\theta_0)\|_2 = o_p(1)$ under the assumed dependence structure. Controlling the Euclidean norm can be based on the following inequality:
  \begin{equation}\label{eq:FiniteMaximumL2}
  \|\hat{\mathcal{Z}}_n(\theta_0)\|_2 = \sup_{\norm{\nu}_2 \le 1}\,\nu^{\top}\hat{\mathcal{Z}}_n(\theta_0) \le 2\max_{\nu\in\mathcal{N}_{1/2}}\,\nu^{\top}\hat{\mathcal{Z}}_n(\theta_0),
  \end{equation}
  where $\mathcal{N}_{1/2}\subset B(0, 1)$ denotes the $1/2$-covering set of $B(0, 1)$, that is,
  \[
  \sup_{\nu\in B(0, 1)}\inf_{\mu\in \mathcal{N}_{1/2}}\,\norm{\nu - \mu}_2 \le 1/2.
  \]
  From Lemma 4.1 of \cite{Pollard90}, it follows that the cardinality of $\mathcal{N}_{1/2}$, $|\mathcal{N}_{1/2}|$, is bounded by $6^q$. Note that inequality~\eqref{eq:FiniteMaximumL2} is sharp up to the factor of $2$. This inequality shows that the tail bounds on $\nu^{\top}\hat{\mathcal{Z}}_n(\theta_0)$ can be used to control $\|\hat{\mathcal{Z}}_n(\theta_0)\|_2$. 
	\item Prove that
	\[
	\norm{\hat{\mathcal{Q}}_n(\theta_0) - \mathcal{Q}_n(\theta_0)}_{op} = o_p(1).
	\]
	This would imply if $\mathcal{Q}_n(\theta_0)$ is positive definite then $\hat{\mathcal{Q}}_n(\theta_0)$ is also positive definite for sufficiently large $n$. Similar to the Euclidean norm, the operator norm can also be bounded in terms of a finite maximum. By Lemma 2.2 of~\cite{Ver12}, it follows that
  \[
  \|\hat{\mathcal{Q}}_n(\theta_0) - \mathcal{Q}_n(\theta_0)\|_{op} \le 2\max_{\nu\in\mathcal{N}_{1/4}}\,\left|\nu^{\top}\hat{\mathcal{Q}}_n(\theta_0)\nu - \nu^{\top}\mathcal{Q}_n(\theta_0)\nu\right|,
  \]
  where again $\mathcal{N}_{1/4}\subset B(0, 1)$ represents the $1/4$-covering number of $B(0, 1)$ and by Lemma 4.1 of~\cite{Pollard90}, $|\mathcal{N}_{1/4}| \le 12^q$.
\end{enumerate}
The quantities $\nu^{\top}\hat{\mathcal{Z}}_n(\theta_0)$ and $\nu^{\top}\hat{\mathcal{Q}}_n(\theta_0)\nu$ being averages are much easier to study under various dependence settings of interest. Exponential-type tail bounds for averages under independence and functional dependence are given in Theorems A.1 and B.1, respectively, of \cite{Uniform:Kuch18}.

Assumption~\ref{eq:Event} is used in the proof of Theorem~\ref{thm:ConvexMEst} only to prove condition~\eqref{assump:Contraction} in Theorem~\ref{thm:NonSingular}. So, any alternative condition implying~\eqref{assump:Contraction} can be used instead. The assumption on the ratio rather than the difference of Hessians is more appealing since minimizers do not depend on the scaling of objective functions. The function $C(\cdot, \cdot)$ naturally cancels out the scalings and requires much weaker conditions as discussed in Section~\ref{sec:Comparison}. 

The function $C(\cdot, \cdot)$ can be bounded easily for self-concordant type convex functions. Proposition 1 of \cite{bach2010self} bounds $C(\cdot, \cdot)$ for logistic regression and see Proposition 8 of \cite{Sun2017} for a general class of convex functions called generalized self-concordant where the ratio of the Hessians is bounded. Also, see~\cite{karimireddy2018global} for other examples.

One specific corollary of Theorem~\ref{thm:ConvexMEst} in regression analysis is of special interest for our applications. For this result, consider independent random variables $(X_i, Y_i)\in\mathbb{R}^p\times\mathbb{R}$ $(1 \le i\le n)$ and the estimator
\[
\hat{\beta}_n := \argmin_{\theta\in\mathbb{R}^p}\,\frac{1}{n}\sum_{i=1}^n h(X_i)\ell(X_i^{\top}\theta, Y_i),
\]
for some loss function $\ell(\cdot, \cdot)$ convex and twice differentiable in the first argument. Here the ``weight'' $h(\cdot)$ is any function not depending on $\theta$. Observe that if $h(\cdot)$ is not a non-negative function, the objective function is not necessarily convex. Define the target vector
\[
\beta_n := \argmin_{\theta\in\mathbb{R}^p}\,\frac{1}{n}\sum_{i=1}^n \mathbb{E}\left[h(X_i)\ell(X_i^{\top}\theta, Y_i)\right].
\]
The function $h(x)\ell(x^{\top}\theta, y)$ can be changed to any function of the form $\ell(x^{\top}\theta; x, y)$ but for simplicity we restrict to the function above. Let
\[
\ell'(u, y) := \frac{\partial}{\partial t}\ell(t, y)\bigg|_{t = u}\quad\mbox{and}\quad \ell''(u, y) := \frac{\partial}{\partial t}\ell'(t, y)\bigg|_{t = u}.
\]
Define the analogue of the $C$ function,
\[
C(u, y) := \sup_{|s - t| \le u}\,\frac{\ell''(s, y)}{\ell''(t, y)}.
\]
Finally, define the analogues of $\hat{\mathcal{Z}}_n(\cdot),\hat{\mathcal{Q}}_n(\cdot)$,
\begin{align*}
\hat{\mathcal{Z}}_n(\theta) := \frac{1}{n}\sum_{i=1}^n \ell'(X_i^{\top}\theta, Y_i)h(X_i)X_i,\quad&\mbox{and}\quad \hat{\mathcal{Q}}_n(\theta) := \frac{1}{n}\sum_{i=1}^n \ell''(X_i^{\top}\theta, Y_i)h(X_i)X_iX_i^{\top},\\
\delta_{n}(\theta) := \frac{3}{2}\norm{[\hat{\mathcal{Q}}_n(\theta)]^{-1}\hat{\mathcal{Z}}_n(\theta)}_2,\quad&\mbox{and}\quad\mathcal{Q}_n(\theta) := \frac{1}{n}\sum_{i=1}^n \mathbb{E}\left[\ell''(X_i^{\top}\theta, Y_i)h(X_i)X_iX_i^{\top}\right].
\end{align*}
\begin{cor}\label{cor:RegressionEst}
If $\ell(\cdot, \cdot)$ is a twice differentiable function that is convex in the first argument and for some $\beta_0\in\mathbb{R}^p$,
\begin{equation}\label{eq:AssumptionA2}
\max_{1\le i\le n}\,C\left(\norm{X_i}_2\delta_n(\beta_0), Y_i\right) \le \frac{4}{3},
\end{equation}
then there exists a vector $\hat{\beta}_n\in\mathbb{R}^p$ satisfying
\[
\hat{\mathcal{Z}}_n(\hat{\beta}_n) = 0,\qquad \frac{1}{2}\delta_n(\beta_0) \le \|{\hat{\beta}_n - \beta_0}\|_2 \le \delta_n(\beta_0),
\]
and
\begin{align*}
&\norm{\hat{\beta}_n - \beta_0 + [{\mathcal{Q}}_n(\beta_0)]^{-1}\hat{\mathcal{Z}}_n(\beta_0)}_2\\
&\qquad\le \left[\max_{1\le i\le n}\,C(\norm{X_i}_2\delta_n(\beta_0), Y_i) - 1 + \norm{[\mathcal{Q}_n(\beta_0)]^{-1}\hat{\mathcal{Q}}_n(\beta_0) - I}_{op}\right]\delta_n(\beta_0).
\end{align*}
\end{cor}
\begin{proof} See Appendix~\ref{appsec:GLM} for a proof.
\end{proof}
\begin{exam}[Linear Models]
\emph{In the following, we bound the function $C$ in case of several linear models. Since Corollary~\ref{cor:RegressionEst} does not require any specific stochastic or model assumptions, the following examples also do not require any ``correct'' modeling assumptions and are deterministic in nature.}
\begin{enumerate}
  \item Linear Regression: \emph{In case of ordinary least squares (OLS) linear regression, the loss function is given by $\ell(t, y) = (t - y)^2$ and the weight function is identically 1. So, $\ell''(u, y) = 2$ and $C(u, y) = 1$ for all $u, y$. This implies that the assumption~\ref{eq:Event} always holds. This is an expected result since the least square estimator satisfies}
  \[
  \frac{1}{n}\sum_{i=1}^n X_i\left(Y_i - X_i^{\top}\hat{\beta}_n\right) = 0,
  \]
  \emph{and subtracting $\beta_0$ from $\hat{\beta}_n$ implies that}
  \[
  \left(\frac{1}{n}\sum_{i=1}^n X_iX_i^{\top}\right)\left(\hat{\beta}_n - \beta_0\right) = \frac{1}{n}\sum_{i=1}^n X_i(Y_i - X_i^{\top}\beta_0).
  \]
  \emph{Here $\beta_0$ is the target OLS vector defined by}
  \[
  \beta_0 := \argmin_{\theta\in\mathbb{R}^p}\,\frac{1}{n}\sum_{i=1}^n \mathbb{E}\left[(Y_i - X_i^{\top}\theta)^2\right].
  \]
  \emph{This proves that}
  \[
  \|\hat{\beta}_n - \beta_0\|_2 = \frac{2\delta_n(\beta_0)}{3}.
  \]
  \emph{The second conclusion of Corollary~\ref{cor:RegressionEst} provides better information:}
  \[
  \norm{\hat{\beta}_0 - \beta_0 - \frac{1}{n}\sum_{i=1}^n \Sigma_n^{-1}X_i(Y_i - X_i^{\top}\beta_0)}_2 \le \norm{\Sigma_n^{-1}\hat{\Sigma}_n - I}_{op}\delta_n(\beta_0),
  \]
  \emph{where}
  \[
  \hat{\Sigma}_n := \frac{1}{n}\sum_{i=1}^n X_iX_i^{\top},\quad\mbox{and}\quad \Sigma_n := \frac{1}{n}\sum_{i=1}^n \mathbb{E}\left[X_iX_i^{\top}\right].
  \]
  \emph{Details on how to bound $\delta_n(\beta_0)$ in case of independent/functionally dependent data were provided in \cite{Uniform:Kuch18}.}
  \item Poisson Regression: \emph{In case of Poisson regression, the loss function is $\ell(t, y) = \exp(t) - yt$ and the weight function is identically 1. So, $\ell''(t, y) = \exp(t)$. This implies that $C(u, y) = \exp(u)$. The event~\eqref{eq:AssumptionA2} is equivalent to}
  \[
  \max_{1\le i\le n}\,\norm{X_i}_2\delta_n(\beta_0) \le \log\left(4/3\right).
  \]
  \emph{On this event,}
  \[
  \max_{1\le i\le n}C\left(\norm{X_i}_2\delta_n(\beta_0)\right) - 1 \le \frac{4\delta_n(\beta_0)}{3}\max_{1\le i\le n}\norm{X_i}_2.
  \]
  \emph{Thus, Corollary~\ref{cor:RegressionEst} implies that there exists $\hat{\beta}_n\in\mathbb{R}^p$ such that}
  \begin{align*}
  &\norm{\hat{\beta}_n - \beta_0 - \frac{1}{n}\sum_{i=1}^n [\mathcal{Q}_n(\beta_0)]^{-1}X_i\left[Y_i - \exp(X_i^{\top}\beta_0)\right]}_2\\ &\qquad\le \left[\frac{4\delta_n(\beta_0)}{3}\max_{1\le i\le n}\norm{X_i}_2 + \norm{[\mathcal{Q}_n(\beta_0)]^{-1}\hat{\mathcal{Q}}_n(\beta_0) - I}_{op}\right]\delta_n(\beta_0),
  \end{align*}
  \emph{where}
  \[
  \hat{\mathcal{Q}}_n(\beta) := \frac{1}{n}\sum_{i=1}^n X_iX_i^{\top}\exp\left(X_i^{\top}\beta\right),\quad\mbox{and}\quad \mathcal{Q}_n(\beta) = \mathbb{E}\left[\hat{\mathcal{Q}}_n(\beta)\right].
  \]
  \emph{Since $\hat{\mathcal{Q}}_n(\beta_0)$ is a Gram matrix based on random vectors $X_i\exp(X_i^{\top}\beta_0/2), 1\le i\le n$, the results of \cite{Uniform:Kuch18} can still be applied to show that $\hat{\mathcal{Q}}_n(\beta_0)$ is close to $\mathcal{Q}_n(\beta_0)$.}
  \item Logistic and Negative Binomial Regression: \emph{In case of logistic regression, the loss function is given by}
  \[
  \ell(u, y) = \log(1 + \exp(u)) - yu,
  \]
  \emph{and the weight function is identically 1. It is easy to show that}
  \[
  \ell''(u, y) = \frac{\exp(u)}{(1 + \exp(u))^2},\quad\mbox{and}\quad C(u, y) = \sup_{|s - t| \le u}\,\frac{\exp(s)(1 + \exp(t))^2}{(\exp(s) + 1)^2\exp(t)}.
  \] 
  \emph{Since $\exp(s) \le \exp(u)\exp(t)$ for all $s, t$ satisfying $|s - t| \le u$, it follows that}
  \[
  C(u, y) \le \sup_{|s - t| \le u}\frac{\exp(s)}{\exp(t)}\sup_{|s - t| \le u}\frac{(\exp(s) + 1)^2}{(\exp(t) + 1)^2} \le \exp(3u).
  \]
  \emph{For the case of negative binomial regression (with parameter $\alpha > 0$), the loss function is}
  \[
  \ell(u, y) = -yu + [y + 1/\alpha]\log(1 + \alpha\exp(u)),
  \]
  \emph{and the weight function is identically 1. So,}
  \[
  \ell''(u, y) = \frac{\alpha[y + 1/\alpha]\exp(u)}{(\alpha\exp(u) + 1)^2},\;\mbox{and}\;C(u, y) = \sup_{|s - t| \le u}\frac{\exp(s)(\alpha\exp(t) + 1)^2}{(\alpha\exp(s) + 1)^2\exp(t)}.
  \]
  \emph{Similar to the logistic regression case, we get $C(u, y) \le \exp(3u).$ Therefore, condition~\eqref{eq:AssumptionA2} becomes}
  \[
  \max_{1\le i\le n}\norm{X_i}_2\delta_n(\beta_0) \le \frac{\log(4/3)}{3}.
  \]
  \emph{Hence calculations similar to the Poisson regression case still hold true.}
\end{enumerate}
\end{exam}
In the examples above, we have controlled the function $C(u, y)$ for some widely used convex examples. When the loss function $\ell(\cdot, y)$ is strongly convex, then $C(u, y) - 1$ can be bounded by $\mathfrak{C}\sup\{|\ell''(s, y) - \ell''(t, y)|:\,|s - t|\le u\}$ for some constant $\mathfrak{C} > 0$. It should, however, be noted that $C(u, y)$ may not be a bounded function even if the function $\ell(\cdot, y)$ is strictly convex. A possible example is probit regression. In this case the approch used in Section~\ref{sec:NonConvex} works easily.
\subsection{Comparison with assumptions in the literature}\label{sec:Comparison}
Results similar to Corollary~\ref{cor:RegressionEst} were presented in \citet[Theorem 1]{Li2017}, \citet[Theorem 1]{Liang2012}, \citet[Corollary 3]{Neg09} and \citet[Example 3]{He2000}. In these papers the authors assume a lower bound on the second order curvature, that is,
\[
\inf_{1\le i\le n}\inf_{\norm{\theta - \beta_0}_2 \le \varepsilon} \ell''(Y_i, X_i^{\top}\theta) \ge \kappa > 0\quad\mbox{for some (small enough)}\quad \varepsilon > 0.
\]
This is a difficult assumption to be satisfied in case of increasing dimension since $\ell''(y,u)$ converges to zero as $u \to -\infty$ usually and often $|X_i^{\top}\beta_0|$ itself grows with the dimension. This hurdle poses certain unnecessary rate constraints on the dimension. In contrast our assumption is based on difference meaning $X_i^{\top}(\theta - \beta_0)$ which can be expected to be small as long as $\norm{\theta - \beta_0}_2$ is small even with increasing dimension. See the discussion surrounding equation (1.8) and Theorem 2.4 of \cite{Bose03} for related ratio-type assumptions.

It is clear from Corollary~\ref{cor:RegressionEst} that the function $C(u, y)$ plays a very important role in the existence and determining the rate of convergence of the estimator. The following proposition (proved in Appendix~\ref{appsec:NonSingular}) allow construction of new loss functions with a control on the $C(\cdot, \cdot)$ function.
\begin{prop}\label{prop:Convexity}
Suppose $\mathcal{C}_T$ (indexed by a non-negative function $T(\cdot, \cdot)$) is the class of all loss functions $L(\cdot, \cdot)$ convex in the first argument and satisfying
\[
\sup_{\norm{\theta_1 - \theta_2} \le u}\,\sup_{e\in\mathbb{R}^q:\,\norm{e}_2 = 1}\frac{e^{\top}\nabla_2 L(\theta_1, w)e}{e^{\top}\nabla_2 L(\theta_2, w)e} \le T(u, w)\quad\mbox{for all}\; u \ge 0\;\mbox{and}\; w.
\]
Then $\mathcal{C}_T$ is a convex cone.
\end{prop}
\section{Deterministic Inequality for Cox Proportional Hazards Model}\label{sec:Cox}
One of the most widely used models in survival analysis is the celebrated Cox proportional hazards model. The partial log-likelihood of the Cox model even though not an average can be dealt using our theory. The analysis in this section is related to the discussion in Section 6 of \cite{HjortPollard}. The usual Cox regression model for possibly censored lifetimes with covariate information is as follows: The individuals have independent lifetimes $T_1^0, \ldots, T_n^0$ and the $i$-th subject has hazard rate 
\begin{equation}\label{eq:CorrectCox}
\lambda_i(s) := \lambda(s)\exp\left(\beta^{\top}_0X_{i,s}\right),
\end{equation}
for some vector $\beta_0$, some baseline hazard function $\lambda(\cdot)$ and $i$-th subject covariate $X_{i,s}\in\mathbb{R}^p$. The classical Cox model has a fixed set of covariates not depending on time $s$ and here they are allowed to depend on time. There is a possibly interfering censoring time $C_i$ leaving the observables to be
\[
T_i = \min\{T_i^0, C_i\}\quad\mbox{and}\quad\delta_i = \mathbbm{1}\{T_i^0 \le C_i\}.
\]
Consider the risk indicator function $Y_{i,s} = \mathbbm{1}\{T_i \ge s\}$, and the counting process $N_i$ with mass $\delta_i$ at $T_i$, that is,
\[
dN_i(s) := \mathbbm{1}\{T_i \in [s, s+ds], \delta_i = 1\}.
\]
The log-partial likelihood is then given by
\[
G_n(\beta) := \sum_{i = 1}^n \int_0^{\infty} \left\{\beta^{\top}X_{i,s} - \log R_n(s, \beta)\right\}dN_i(s),
\]
where
\[
R_n(s, \beta) := \sum_{i = 1}^n Y_i(s)\exp\left(\beta^{\top}X_{i,s}\right).
\]
The Cox estimator is the value $\hat{\beta}_n$ that maximizes the log-partial likelihood. Even though the motivation above is through a correct model~\eqref{eq:CorrectCox}, we do not make any such assumptions and prove a purely deterministic result. 
Define for $\beta\in\mathbb{R}^{p}$,
\[
\hat{\mathcal{L}}_{n}(\beta) := \sum_{i=1}^n \int_0^{\infty} H_{1}(X_{i,s})\left\{\log R_{n}(s, \beta) - \beta^{\top}X_{i,s}\right\}dN_i(s),
\]
where
\[
R_{n}(s, \beta) := \sum_{i=1}^n H_{2}(X_{i,s})Y_i(s)\exp\left(\beta^{\top}X_{i,s}\right).
\]
The objective function $\hat{\mathcal{L}}_{n}(\cdot)$ is a generalization of $G_n(\cdot)$ allowing for two functions $H_{1}(\cdot)$ and $H_{2}(\cdot)$ that can be used to down-weight outliers in the covariate space. Note that this generalization does not change the convexity property of the objective function. The Cox estimator based on $\hat{\mathcal{L}}_{n}(\cdot)$ is given by
\[
\hat{\beta}_{n} := \argmin_{\theta\in\mathbb{R}^{p}}\,\hat{\mathcal{L}}_{n}(\theta).
\]
Define for $\beta\in\mathbb{R}^{p}$,
\[
\hat{\mathcal{Z}}_{n}(\beta) := \sum_{i=1}^n \int_0^{\infty} H_{1}(X_{i,s})\left\{\frac{\dot{R}_{n}(s, \beta)}{R_{n}(s, \beta)} - X_{i,s}\right\}dN_i(s),
\]
where
\[
\dot{R}_{n}(s, \beta) := \frac{\partial R_{n}(s, \beta)}{\partial\beta}.
\]
Define the Jacobian as $\hat{\mathcal{Q}}_{n}(\beta) := \nabla \hat{\mathcal{Z}}_{n}(\beta).$ Finally define for any $\beta_0\in\mathbb{R}^p$, 
\[
\bar{X}_{n,s}(\beta_0) := \frac{\sum_{i=1}^n X_{i,s}H_{2}(X_{i,s})Y_i(s)\exp\left(\beta^{\top}_{0}X_{i,s}\right)}{R_{n}(s, \beta_{0})}.
\]
\begin{thm}\label{thm:CoxModel}
Set for any target vector $\beta_0\in\mathbb{R}^p$,
\begin{align*}
\mu_{n}(s) := \max_{1\le i\le n}\norm{X_{i,s} - \bar{X}_{n,s}(\beta_0)}_2,\quad\mbox{and}\quad\delta_{n}(\beta_0) := \frac{3}{2}\norm{[\hat{\mathcal{Q}}_{n}(\beta_{0})]^{-1}\hat{\mathcal{Z}}_{n}(\beta_{0})}_2.
\end{align*}
If
\begin{equation}\label{assump:CoxRegression}
\sup_{0 \le s < \infty} \mu_{n}(s)\delta_{n}(\beta_0) \le \frac{1}{16},
\end{equation}
then there exists a vector $\hat{\beta}_{n}\in\mathbb{R}^{p}$ satisfying
\begin{equation}\label{eq:CoxConsis}
\hat{\mathcal{Z}}_{n}(\hat{\beta}_{n}) = 0,\quad\mbox{and}\quad \frac{1}{2}\delta_{n}(\beta_0) \le \|\hat{\beta}_{n} - \beta_{0}\|_2 \le \delta_{n}(\beta_0).
\end{equation}
Furthermore, 
\begin{equation}\label{eq:CoxAsympLinear}
\begin{split}
\norm{\hat{\beta}_{n} - \beta_{0} + [\hat{\mathcal{Q}}_{n}(\beta_{0})]^{-1}\hat{\mathcal{Z}}_{n}(\beta_{0})}_2 &~\le~ 8e^{1/4}\delta_{n}^2(\beta_0)\sup_s\mu_{n}(s).
\end{split}
\end{equation}
\end{thm}
\section{Deterministic Inequalities for Non-convex \textit{M}-estimators}\label{sec:NonConvex}
In previous sections we have proved the applicability of Theorem~\ref{thm:NonSingular} for convex loss functions. However, Theorem~\ref{thm:NonSingular} does not require ``monotonicity''\footnote{Derivatives of differentiable one-dimensional convex functions are non-decreasing.} of the function $f(\cdot)$. In this section, we provide one specific non-convex example, namely, non-linear regression.
\subsection{Least Squares Non-linear Regression}
For a motivation of non-linear regression, consider the problem of binary linear classification based on $n$ paris $(X_1, Y_1), \ldots, (X_n, Y_n)$ with $Y_i\in\{0,1\}$ and $X_i\in\mathbb{R}^p$. In this model, the quantity of interest is the conditional probability of $Y_i$ given $X_i$. Suppose $\mathbb{P}(Y_i = 1|X_i = x) = \sigma(x^{\top}\theta_0)$ with $\theta_0\in\mathbb{R}^p$ and a function $\sigma:\mathbb{R}\to[0,1]$. Since this implies $\mathbb{E}[Y_i|X_i = x] = \sigma(x^{\top}\theta_0)$, one possible estimator of $\theta_0$ is obtained by minimizing the squared error loss:
\[
\frac{1}{n}\sum_{i=1}^n \left(Y_i - \sigma(X_i^{\top}\theta)\right)^2,
\]
with respect to $\theta\in\mathbb{R}^p$. It is easy to see that the loss function above is, in general, non-convex. In constrast to convex losses (e.g., hinge or logistic), non-convex loss functions as above have better classification accuracy in various scenarios; see~\cite{nguyen2013algorithms} and~\cite{mei2016landscape}.

As a generalization consider the observations $(X_i, Y_i)\in\mathbb{R}^p\times\mathbb{R}$ for $1\le i\le n$ and the loss function
\[
F_n(\theta) := \frac{1}{n}\sum_{i=1}^n \left(Y_i - g(\theta^{\top}X_i)\right)^2,
\]
for a known function $g(\cdot)$ that is twice differentiable and bounded. (We do not restrict to $Y_i\in\{0, 1\}$.) To prove a deterministic inequality for the stationary points of $F_n(\cdot)$, we use the following assumption:
\begin{description}
  \item[\namedlabel{eq:NonLinear}{(NR)}] The function $g(\cdot)$ is twice differentiable and there exists functions $C_0(\cdot)$, $C_1(\cdot)$, $C_2(\cdot)$ such that for some $\alpha \in (0, 1]$ and any $x, \theta_1, \theta_2$,
  \begin{align*}
  \left|g(x^{\top}\theta_1) - g(x^{\top}\theta_2)\right| &\le C_0(x)\norm{\theta_1 - \theta_2}_2,\\
  \left|g'(x^{\top}\theta_1) - g'(x^{\top}\theta_2)\right| &\le C_1(x)\norm{\theta_1 - \theta_2}_2,\mbox{ and}\\
  \left|g''(x^{\top}\theta_1) - g''(x^{\top}\theta_2)\right| &\le C_2(x)\norm{\theta_1 - \theta_2}_2^{\alpha}.
  \end{align*}
\end{description}

Assumption~\ref{eq:NonLinear} is satisfied for many classical activation functions with $\alpha = 1$ (for example, logistic function). Another important example satisfying assumption~\ref{eq:NonLinear} is the phase retrieval problem where $g(t) = t^2$; see~\cite{yang2017misspecified} for recent developments. From the proof of Corollary~\ref{cor:NonLinearRegression}, it follows that assumption~\ref{eq:NonLinear} can be relaxed to $\theta_1, \theta_2\in B_r(\theta_0)$ for some $r > 0$.  Define for any $\theta\in\mathbb{R}^p$, $\delta_n(\theta) := 1.5\norm{(\nabla_2F_n(\theta))^{-1}\nabla F_n(\theta)}_2$ and
\begin{align*}
L_2(\theta) &:= \norm{\frac{2}{n}\sum_{i=1}^nC_1^2(X_i)\left(\nabla_2 F_n(\theta)\right)^{-1}X_iX_i^{\top}}_{op},\\
L_{1 + \alpha}(\theta) &:= \norm{\frac{2}{n}\sum_{i=1}^n C_0(X_i)C_2(X_i)\left(\nabla_2 F_n(\theta)\right)^{-1}X_iX_i^{\top}}_{op},\\
L_1(\theta) &:= \norm{\frac{2}{n}\sum_{i=1}^n \left\{2C_1(X_i)|g'(X_i^{\top}\theta)| + C_0(X_i)|g''(X_i^{\top}\theta)|\right\}\left(\nabla_2 F_n(\theta)\right)^{-1}X_iX_i^{\top}}_{op},\\
L_{\alpha}(\theta) &:= \norm{\frac{2}{n}\sum_{i=1}^n C_2(X_i)|Y_i - g(X_i^{\top}\theta)|\left(\nabla_2 F_n(\theta)\right)^{-1}X_iX_i^{\top}}_{op}.
\end{align*}
The following result shows the existence of a solution that satisfies an asymptotic expansion. The proof (in Appendix~\ref{appsec:NonConvex}) verifies the assumptions of Theorem~\ref{thm:NewtonStepDifferentiable}. 
\begin{cor}\label{cor:NonLinearRegression}
Under assumption~\ref{eq:NonLinear}, for any $\theta_0$ satisfying 
\begin{equation}\label{eq:DerivativeSmall}
\delta_n(\theta_0) \le \min\left\{(12L_j(\theta_0))^{-1/j}:\,j\in\{\alpha, 1, 1 + \alpha, 2\}\right\},
\end{equation}
there exists a unique solution $\hat{\theta}_n$ of $\nabla F_n(\theta) = 0$ in $B(\theta_0, \delta_n(\theta_0))$ and this solution $\hat{\theta}_n$ satisfies
\[
\norm{\hat{\theta}_n - \theta_0 + (\nabla_2F_n(\theta_0))^{-1}\nabla F_n(\theta_0)}_2 \le \omega(\delta_n(\theta_0))\delta_n(\theta_0),
\]
where for $r\ge 0,$
\[
\omega(r) := L_2(\theta_0)r^2 + L_{1 + \alpha}(\theta_0)r^{1 + \alpha} + L_1(\theta_0)r + L_{\alpha}(\theta_0)r^{\alpha}.
\]
\end{cor}

Corollary~\ref{cor:NonLinearRegression} can be compared to Theorem 4 of~\cite{mei2016landscape}. As described in Remark~\ref{rem:Landscape}, if $\theta^{(j)}, 1\le j\le K$ denote the solutions of $\nabla \mathbb{E}[F_n(\theta)] = 0$ and $\mathbb{E}[F_n(\theta)]$ is a Morse function\footnote{A function $R(\theta)$ is said to be a Morse function if for any $\theta_0$ satisfying $\nabla R(\theta_0) = 0$, the Hessian $\nabla_2R(\theta_0)$ is invertible.}, then by Corollary~\ref{cor:NonLinearRegression} the sample estimating equation $\nabla F_n(\theta) = 0$ also has solutions near $\theta^{(j)}$ for each $1\le j\le K$. Further, the result applies for a larger class of link functions $g$ and allows for dependent observations. Also, note that we do not need to verify uniform in $\theta$ control of the gradient/Hessian which was required in~\cite{mei2016landscape}.
\section{Deterministic Inequalities for Equality Constrained Problems}\label{sec:Constrained}
In the context of linear models, hypothesis tests related to linear combinations of the coefficients form an important component of applied analysis. For instance, it is of interest to know if the treatment effect is more than that of the control when both effects are measured in terms of the coefficients in the linear model. See Section 1.4 of \cite{amemiya1985advanced} for details.

Consider the problem of minimizing a twice differentiable function $F_n(\beta)$ subject to $A\beta = b$, for some matrix $A\in\mathbb{R}^{d\times p}$ of full row rank and vector $b\in\mathbb{R}^d$. A vector $\beta^{\star}\in\mathbb{R}^p$ is a minimizer of this constrained problem \emph{only if} there exists a vector $\nu^{\star}\in\mathbb{R}^d$ such that the following KKT equations are satisfied:
\begin{equation}\label{eq:KKTCondition}
A\beta^{\star} = b,\quad\mbox{and}\quad \nabla F_n(\beta^{\star}) + A^{\top}\nu^{\star} = 0.
\end{equation}
If, in addition, the function $F_n(\cdot)$ is convex, then the KKT equations are also sufficient. Some commonly used convex examples of $F_n(\beta)$ are
\begin{equation}\label{eq:GLMsEqualityConstrained}
F_n(\beta) = \frac{1}{n}\sum_{i=1}^n \left\{\psi(X_i^{\top}\beta) - Y_iX_i^{\top}\beta\right\},
\end{equation}
with $\psi(t)\in\{t^2/2, \log(1 + \exp(t)), \exp(t)\}$. A non-convex example of $F_n(\cdot)$ is
\begin{equation}\label{eq:NonConvexEqualityConstrained}
F_n(\beta) = \frac{1}{n}\sum_{i=1}^n \left(Y_i - g(X_i^{\top}\beta)\right)^2,
\end{equation}
with $g(\cdot)$ satisfying assumption~\ref{eq:NonLinear}.

The following result proves the existence and an expansion for a local minimizer in equality constrained problems. For this result, define for $\beta\in\mathbb{R}^p,$ and $\nu\in\mathbb{R}^d$,
\[
\delta_{n}(\beta,\nu) := 1.5\left(1 + \norm{(A[\nabla_2F_n(\beta)]^{-1}A^{\top})^{-1}A}_{op}\right)\norm{[\nabla_2F_n(\beta)]^{-1}(\nabla F_n(\beta) + A^{\top}\nu)}_2.
\]
\begin{cor}\label{cor:EqualityConstrained}
Fix vectors $\nu_0\in\mathbb{R}^d$ and $\beta_0\in\mathbb{R}^p$ such that $A\beta_0 = b$. Suppose $F_n(\cdot)$ is a twice differentiable function. If there exist constants $L \ge 0$ and $\alpha\in(0, 1]$, such that for all $\beta\in B(\beta_0, (3L)^{-1/\alpha})$,
\begin{equation}\label{eq:CheckHolderContinuityEuality}
\norm{[\nabla_2F_n(\beta_0)]^{-1}(\nabla_2F_n(\beta) - \nabla_2F_n(\beta_0))}_{op} \le L\norm{\beta - \beta_0}_2^{\alpha},
\end{equation}
and $\delta_n(\beta_0,\nu_0) \le (3L)^{-1/\alpha},$ then there exists a vector $(\hat{\beta}_n, \hat{\nu}_n)\in\mathbb{R}^p\times\mathbb{R}^d$ solving the KKT equations~\eqref{eq:KKTCondition} and the vector $\hat{\beta}_n$ satisfies the expansion
\begin{equation}\label{eq:ExpasionEualityConstrained}
\norm{\hat{\beta}_n - \beta_0 - [J_n(\beta_0)]^{-1}(\nabla F_n(\beta_0) + A^{\top}\nu_0)}_2 \le L[\delta_n(\beta_0, \nu_0)]^{1 + \alpha}.
\end{equation}
Here
\[
J_n(\beta_0) := [\nabla_2F_n(\beta_0)]\left(I - [\nabla_2F_n(\beta_0)]^{-1}A^{\top}(A[\nabla_2F_n(\beta_0)]^{-1}A^{\top})^{-1}A\right)^{-1}.
\]
\end{cor}
Note that condition~\eqref{eq:CheckHolderContinuityEuality} is verified for the examples~\eqref{eq:GLMsEqualityConstrained} and~\eqref{eq:NonConvexEqualityConstrained} in Sections~\ref{sec:GLM} and~\ref{sec:NonConvex}. For an application of this result in statistical context, one would take $\beta_0\in\mathbb{R}^p$ as the minimizer of $\mathbb{E}[F_n(\beta)]$ subject to $A\beta = b$. The vector $\nu_0\in\mathbb{R}^{d}$ would be the vector satisfying the ``population'' KKT equations
\[
A\beta_0 = b\quad\mbox{and}\quad \mathbb{E}[\nabla F_n(\beta_0)] + A^{\top}\nu_0 = 0.
\]
This implies that $\nabla F_n(\beta_0) + A^{\top}\nu_0$ is a mean zero random vector and so, the expansion~\eqref{eq:ExpasionEualityConstrained} implies asymptotic normality of the (properly normalized) local minimizer $\hat{\beta}_n$. It is easy to generalize Corollary~\ref{cor:EqualityConstrained} when the linear equality constraint $A\beta = b$ is replaced by a non-linear constraint $G(\beta) = 0$ (which makes the problem non-convex even if $F_n(\beta)$ is convex).
\begin{rem}\,(General Constraints)
It is of considerable interest to extend Corollary~\ref{cor:EqualityConstrained} to $M$-estimation problems with more general inequality/abstract constraints. It is not clear if a useful deterministic inequality is possible. For example, consider the minimization problem
\[
\min_{\beta}\,F_n(\beta)\quad\mbox{subject to}\quad\begin{cases}G_n(\beta) = 0,&\\
H_n(\beta) \ge 0.
\end{cases}
\]
Suppose the functions $F_n(\cdot), G_n(\cdot), H_n(\cdot)$ are twice differentiable. Define the Lagrangian function
\[
\mathcal{L}_n(\beta, \lambda, \mu) := F_n(\beta) - \lambda^{\top}G_n(\beta) - \mu^{\top}H_n(\beta).
\]
A vector $\beta^{\star}$ is a (local) minimizer only if there exist $\lambda^{\star}$ and $\mu^{\star}$ such that
\begin{equation}\label{eq:KKTConditionInequality}
\begin{split}
\nabla_{\beta}\mathcal{L}_n(\beta^{\star}, \lambda^{\star}, \mu^{\star}) &= 0,\quad G_n(\beta^{\star}) = 0,\\
H_n(\beta^{\star}) \ge 0,\quad\mu^{\star} &\ge 0,\quad
H_n^{\top}(\beta^{\star})\mu^{\star} = 0.
\end{split}
\end{equation}
The inequalities above can be converted to equalities as follows. Define the function $M(u, v) = \sqrt{u^2 + v^2} - u - v$ for any two vectors $u, v$ (Here $\sqrt{u^2 + v^2}$ is evaluated as a componentwise operation). Then the last three inequalities of the KKT conditions can be equivalently written as
\[
M(H_n(\beta^{\star}), \mu^{\star}) = 0.
\]
The function $M(\cdot, \cdot)$ is known in mathematical programming literature as the Fischer–Burmeister function. Thus the revised KKT conditions can be written as
\begin{equation}\label{eq:KKTConditionInequalityV2}
\begin{split}
\nabla_{\beta}\mathcal{L}_n(\beta^{\star}, \lambda^{\star}, \mu^{\star}) = 0,\quad G_n(\beta^{\star}) = 0,\quad\mbox{and}\quad
M(H_n(\beta^{\star}), \mu^{\star}) = 0.
\end{split}
\end{equation}
The advantage of~\eqref{eq:KKTConditionInequalityV2} is that there are only equations and no inequalities. However, the function $M(H_n(\beta), \mu)$ is not Fr{\'e}chet differentiable but only $B$-differentiable (or semi-smooth). There are semilocal convergence results for Newton's method available in this respect; see~\cite{chen1997verification} and~\cite{wang2008extensions}. For a general treatment of variational inequality problems (VIPs), see~\cite{MR3289054}. But explicit application of these results require certain complimentary qualification conditions that make their usefulness unclear as a general solution; see~\cite{klatte1987note}, ~\cite{dupavcova1991non}, and~\cite{wang2000approximate}. 
\end{rem}
\section{Applications of the Deterministic Inequalities}\label{sec:Applications}
In the previous sections, we have proved deterministically that the estimator normalized around the target behaves like an average when the objetive function is an average. Averages are statistician's friend: most of statistical inference is based on the fact that averages are close to being normally distributed and can be bootstrapped under various dependence structures of interest. In the following subsections, we provide applications of the deterministic inequalities for subsampling/cross-validation methods and two problems related to post-selection inference.  
\subsection{Cross-validation and Subsampling}\label{subsec:CrossValidation}
In this section, we consider applications of the deterministic inequalities in understanding estimators computed based on a subset of the data. Two specific statistical methods that consider estimators based on a subset are cross-validation (CV) and subsampling. Leave-one-out CV predicts the response based on estimator computed using $n-1$ observations. In subsampling with a subsample size $b = b_n$, estimators computed with $b_n$ observations are compared to the one with $n$ observations. Leave-one/$k$-out CV is a popular method for estimating the out-of-sample prediction risk of a model and subsampling is useful in construction of asymptotic confidence intervals. Similar subset estimators appear in the case of delete-d-jackknife. See~\cite{stone1977asymptotic}, \cite{shao1993linear},~\cite{politis1999subsampling} and~\cite{shao1989general} for a detailed discussion of these methods. 

For the result in this section, we consider the setting of Theorem~\ref{thm:ConvexMEst}. The observations are $W_1, W_2, \ldots, W_n$. Define the estimator $\hat{\theta}_n$ as a solution of
\begin{equation}\label{eq:ThetaHat}
\hat{\mathcal{Z}}_n(\theta) = \frac{1}{n}\sum_{i=1}^n \nabla L(\theta; W_i) = 0.
\end{equation}
For simplicity, we first consider the leave-one-out estimator and then consider leave-$k$-out estimator. For any $1\le I\le n$, define the estimator $\hat{\theta}_{-I}$ as a solution of
\[
\frac{1}{n-1}\sum_{\substack{1\le i\le n,\,i\neq I}} \nabla L(\theta, W_i) = 0. 
\]
Under the condition~\eqref{eq:LOOCVCondition} of Corollary~\ref{cor:LeaveOneOut} (below) the existence of $\hat{\theta}_{-I}$ follows from Theorem~\ref{thm:ConvexMEst}. Also, define for $1\le I\le n$,
\[
\delta_{I,n} := \frac{n^{-1}\|\hat{\mathcal{Q}}_n^{-1}\nabla L(\hat{\theta}_n, W_I)\|_2}{1 - n^{-1}\|\hat{\mathcal{Q}}_n^{-1}\nabla_2L(\hat{\theta}_n, W_I)\|_{op}},\quad\mbox{where}\quad\hat{\mathcal{Q}}_n := \frac{1}{n}\sum_{i=1}^n \nabla_2L(\hat{\theta}_n, W_i).
\]
Applying Theorem~\ref{thm:ConvexMEst} for the estimator $\hat{\theta}_{-I}$ and target $\hat{\theta}_n$, we get the following result, a detailed proof of which can be found in Appendix~\ref{appsec:Applications}.
\begin{cor}\label{cor:LeaveOneOut}
Consider the loss function $L(\cdot, \cdot)$ as in assumption~\ref{eq:ConvexityL}. If $\delta_{I,n} \ge 0$ for all $1\le I \le n$ and 
\begin{equation}\label{eq:LOOCVCondition}
\max_{1\le i \neq I\le n}\,C\left(1.5\delta_{I,n}, W_i\right) \le \frac{4}{3},
\end{equation}
then for all $1\le I\le n$,
\begin{equation}\label{eq:LOOCVExpansion}
\begin{split}
&\norm{\hat{\theta}_{-I} - \hat{\theta}_n - n^{-1}\hat{\mathcal{Q}}_n^{-1}\nabla L(\hat{\theta}_n, W_I)}_2\\ 
&\qquad\le \frac{3\delta_{I,n}}{2}\left[\max_{\substack{1\le i\neq I\le n}}\,C(1.5\delta_{I,n}, W_i) - 1 + n^{-1}\norm{\hat{\mathcal{Q}}_n^{-1}\nabla_2L(\hat{\theta}_n, W_I)}_{op}\right].
\end{split}
\end{equation}
\end{cor}
\begin{rem}\,(Comments on the approximation rate)\label{rem:LOOCVRate}
Corollary~\ref{cor:LeaveOneOut} shows that the difference between $\hat{\theta}_{-I}$ and $\hat{\theta}_n$ can be bounded in terms of quantities computable based on the sample of $n$ observations. This is, indeed, expected since $\hat{\theta}_{-I}$ and $\hat{\theta}_n$ are computable based on the sample of $n$ observations. It should be stressed again that Corollary~\ref{cor:LeaveOneOut} is a purely deterministic result and does not require any stochasticity assumptions on the observations. The result can also be readily used to reduce the computational burden of leave-one-out CV. Since $\hat{\mathcal{Q}}_n$ is an average, under most dependence structure would be asymptotically deterministic and so, $\delta_{I,n} = O_p(n^{-1})$ as $n\to\infty$. Therefore, the expansion error in~\eqref{eq:LOOCVExpansion} is in general of order $o_p(n^{-1})$. In fact, if $C(\cdot, w)$ is differentiable at $0$, then the expansion error is of the order $O_p(n^{-2}).$

Following the examples in Section~\ref{sec:GLM} condition~\eqref{eq:LOOCVCondition} can be written explicitly for many common regression examples. A particularly illuminating example is the case of linear regression where condition~\eqref{eq:LOOCVCondition} is satisfied for any set of observations since $C(\cdot, \cdot)\equiv 1$ and the error bound in~\eqref{eq:LOOCVExpansion} becomes
$
1.5n^{-1}\delta_{I,n}\|\hat{\mathcal{Q}}_n^{-1}\nabla_2L(\hat{\theta}_n, W_I)\|_{op}.
$
\end{rem}
\medskip
Leave-one-out CV and delete-1-jackknife are known to have poorer properties in comparison to the leave-k-out CV and delete-d-jackknife methods (see, e.g., \cite{shao1993linear}). For this reason, it is of interest to consider the error obtained in removing more than one observation at a time. The result in this case is also very similar to Corollary~\ref{cor:LeaveOneOut}, albeit with a larger error which is expected. The proof of the following result can be found in Appendix~\ref{appsec:Applications}. Suppose $I$ is a subset of $\{1, 2, \ldots, n\}$ with $|I| < n$ (think $|I| = o(n)$) and consider the estimator $\hat{\theta}_{-I}$ as a solution of
\[
\sum_{1\le i\le n,\,i\notin I}\nabla L(\theta, W_i) = 0.
\]
Here $|I|$ denotes the cardinality of the set $I$. Define
\[
\delta_{I,n} := \frac{n^{-1}\norm{\hat{\mathcal{Q}}_n^{-1}\sum_{i\in I}\nabla L(\hat{\theta}_n, W_i)}_2}{1 - n^{-1}\norm{\hat{\mathcal{Q}}_n^{-1}\sum_{i\in I}\nabla_2L(\hat{\theta}_n, W_i)}_{op}},\quad\mbox{where}\quad\hat{\mathcal{Q}}_n := \frac{1}{n}\sum_{i=1}^n \nabla_2L(\hat{\theta}_n, W_i)
\]
\begin{cor}\label{cor:LeaveMoreOut}
Under the setting of Corollary~\ref{cor:LeaveOneOut}, if $\delta_{I,n} \ge 0$ and
$
C\left(1.5\delta_{I,n}, W_i\right) \le {4}/{3},
$ for all $i\in I^c\cap\{1,2,\ldots,n\}$,
then
\begin{align}
&\norm{\hat{\theta}_{-I} - \hat{\theta}_n - \frac{1}{n}\hat{\mathcal{Q}}_n^{-1}\sum_{i\in I}\nabla L(\hat{\theta}_n, W_i)}_2\label{eq:LeaveMoreOut}\\ 
&\qquad\le \frac{3\delta_{I,n}}{2}\left[\max_{\substack{1\le i\le n,\,i\neq I}}\,C(1.5\delta_{I,n}, W_i) - 1 + n^{-1}\norm{\hat{\mathcal{Q}}_n^{-1}\sum_{i\in I}\nabla_2L(\hat{\theta}_n, W_i)}_{op}\right].
\end{align}
\end{cor}
Clearly, Corollary~\ref{cor:LeaveMoreOut} reduces to Corollary~\ref{cor:LeaveOneOut} when $I$ is a singleton. Even in Corollary~\ref{cor:LeaveMoreOut} one can take maximum over a collection of subsets $I$. Similar to the case in Remark~\ref{rem:LOOCVRate}, under differentiability of $C(\cdot, w)$ at $0$, the expansion error of~\eqref{eq:LeaveMoreOut} is of the order $O_p(|I|^2n^{-2}).$ If $|I| = O(n)$ many observations are removed then it might be better to compare $\hat{\theta}_{-I}$ to $\theta_0$ than to $\hat{\theta}_n$. In case of subsampling or m-of-n bootstrap, the subset of observations are chosen as ``iid sample'' from the empirical distribution. In these cases, a reasonable choice for the target vector is $\hat{\theta}_n$. In case of cross-validation, the subset is not a random sample from the empirical distribution and so $\theta_0$ is a good choice for the target vector.   

It is easy to see that Corollaries~\ref{cor:LeaveOneOut} and~\ref{cor:LeaveMoreOut} can be extended to the case of Cox proportional hazards model and to the other cases given in previous sections.

Since the result of deterministic nature, it is interesting to consider the worst case approximation when considering uniform over all subsets $I$ of size $k$ (with $k$ allowed to change with $n$). For instance if $k = \sqrt{n}$, then the total number of subsets is of the order $O(n^{\sqrt{n}/2})$ which makes it hard to derive a good (polynomial) rate of convergence of the supremum even if the averages have exponential concentration inequalities.


\subsection{Marginal Screening}\label{subsec:MarginalScreen}
In the current era of data science, one is often encountered with a larger number of covariates/predictors in regression data than the number of samples. In this scenario, it has become a common practice to select a subset of covariates either by screening using marginal effects or by some regularized methods. The recent works~\cite{mckeague2015adaptive} and~\cite{wang2018testing} provide a formal testing framework for the existence of any active predictors in linear and quantile regression settings. 

In the linear regression case, the setting is as follows: $(X, Y)\in\mathbb{R}^{p + 1}$ and $(X_i, Y_i), 1\le i\le n$ are iid random vectors and we want to test if the maximal correlation between $X(j)$ (the $j$-th coordinate of $X$) and $Y$ is non-zero. This question in case of non-singular $\mathbb{E}[XX^{\top}]$ is same as testing if there exists any subset of covariates that has linear predictive ability for the response $Y$. To see this let $X(M)$ for $M\subseteq\{1,2,\ldots, p\}$ be a subvector of $X$ with indices in $M$ and define the OLS regression target
\[
\beta_M := \argmin_{\theta\in\mathbb{R}^{|M|}}\,\mathbb{E}\left[(Y_i - X_i^{\top}(M)\theta)^2\right] = \left(\mathbb{E}[X(M)X^{\top}(M)]\right)^{-1}\mathbb{E}\left[X(M)Y\right].
\]
Since the Gram matrix $\mathbb{E}[XX^{\top}]$ is non-singular, $\mathbb{E}[X(M)X^{\top}(M)]$ is non-singular and $\beta_M = 0\in\mathbb{R}^{|M|}$ is equivalent to $\mathbb{E}[X(M)Y] = 0\in\mathbb{R}^{|M|}$. Therefore, 
\[
\beta_M = 0\in\mathbb{R}^{|M|}\quad\mbox{for all}\quad M\subseteq\{1, 2, \ldots, p\},
\]
is equivalent to
\[
\mathbb{E}[X(j)Y] = 0,\quad\mbox{for all}\quad 1\le j\le p. 
\]
In \cite{mckeague2015adaptive}, the authors consider the maximal correlation parameter
\[
\theta_0 := \max_{1\le j\le p}\,\mbox{Corr}\left(X(j), Y\right).
\]
The estimator of $\theta_0$ they consider is 
\[
\hat{\theta}_n := \max_{1\le j\le p}\,\widehat{\mbox{Corr}}\left(X(j), Y\right),
\]
where $\widehat{\mbox{Corr}}$ represents the sample correlation coefficient. It is easy to see that $\hat{\theta}_n$ (properly scaled) is not asymptotically normal and~\cite{mckeague2015adaptive} derive the exact asymptotic distribution along with a resampling procedure to estimate the distribution. 

As an alternative, consider the following inequality
\begin{equation}\label{eq:MaximumDiffIneq}
\left|\hat{\theta}_n - \theta_0\right| \le \max_{1\le j\le p}\left|\widehat{\mbox{Corr}}(X(j), Y) - \mbox{Corr}(X(j), Y)\right|.
\end{equation}
Since $\widehat{\mbox{Corr}}$ is an asymptotically linear estimator, the right hand side above is asymptotically the maximum of an average which can be bootstrapped under various dependence structures. This provides an asymptotically conservative inference in general for the parameter $\theta_0$. (Note, however, that under the null hypothesis $H_0:\theta_0 = 0$ inequality~\eqref{eq:MaximumDiffIneq} is exact and gives valid critical values for Type I error control.)

To elaborate and provide a general framework of marginal screening for $M$-estimators, consider the marginal targets for $1\le j\le p$,
\[
\beta_{n,j} := \argmin_{\theta\in\mathbb{R}}\,\frac{1}{n}\sum_{i=1}^n \mathbb{E}\left[h(X_i(j))\ell(X_i(j)\theta, Y_i)\right],  
\]  
for a twice differentiable convex loss function $\ell(\cdot, \cdot)$ and a non-negative weight function $h(\cdot)$. The estimators for $1\le j\le p$ are given by
\[
\hat{\beta}_{n,j} := \argmin_{\theta\in\mathbb{R}}\,\frac{1}{n}\sum_{i=1}^n h(X_i(j))\ell(X_i(j)\theta, Y_i).
\]
Define for $1\le j\le p$, $\delta_{n,j} := 1.5[\hat{\mathcal{Q}}_{n,j}]^{-1}|\hat{\mathcal{Z}}_{n,j}|$, where
\begin{align*}
\hat{\mathcal{Z}}_{n,j} &:= \frac{1}{n}\sum_{i=1}^n \ell'(X_i(j)\beta_{n,j}, Y_i)h(X_i(j))X_i(j),\\
\hat{\mathcal{Q}}_{n,j} &:= \frac{1}{n}\sum_{i=1}^n h(X_i(j))\ell''(X_i(j)\beta_{n,j}, Y_i)X_i^2(j),\\
\mathcal{Q}_{n,j} &:= \frac{1}{n}\sum_{i=1}^n \mathbb{E}\left[h(X_i(j))\ell''(X_i(j)\beta_{n,j}, Y_i)X_i^2(j)\right].
\end{align*}
The following corollary shows that an asymptotically conservative inference is possible for marginal screening in general $M$-estimators.
\begin{cor}\label{cor:MarginalScreening}
If
\[
\max_{1\le j\le p}\,\max_{1\le i\le n}\,C\left(|X_i(j)|\delta_{n,j}, Y_i\right) \le \frac{4}{3},
\]
then simultaneously for all $j\in\{1, 2, \ldots, p\}$,
\[
\left|\hat{\beta}_{n,j} - \beta_{n,j} - [\mathcal{Q}_{n,j}]^{-1}\hat{\mathcal{Z}}_{n,j}\right| \le \left[\max_{1\le i\le n}\,C\left(|X_i(j)|\delta_{n,j}, Y_i\right) - 1 + \left|\frac{\hat{\mathcal{Q}}_{n,j}}{\mathcal{Q}_{n,j}} - 1\right|\right]\delta_{n,j}.
\]
Furthermore, if
\[
\max_{1\le j\le p}\delta_{n,j} = o_p(1),\quad\mbox{and}\quad \max_{1\le j\le p}\left|\frac{\hat{\mathcal{Q}}_{n,j}}{\mathcal{Q}_{n,j}} - 1\right| = o_p(1),\quad\mbox{ as }\quad n\to\infty,
\]
then
\begin{equation}\label{eq:MaximumDiffIneqCor}
\left|\max_{1\le j\le p}\hat{\beta}_{n,j} - \max_{1\le j\le p}\beta_{n,j}\right| \le (1 + o_p(1))\max_{1\le j\le p}\left|\frac{\hat{\mathcal{Z}}_{n,j}}{\mathcal{Q}_{n,j}}\right|.
\end{equation}
\end{cor}
\begin{proof}
The result follows trivially from Corollary~\ref{cor:RegressionEst}.
\end{proof}
The right hand side of~\eqref{eq:MaximumDiffIneqCor} is the (absolute) maximum of a mean zero average vector and the high-dimensional central limit theorems of~\cite{Cher13,Chern17}, ~\cite{Zhang14} and~\cite{ZhangWu17} provide a Gaussian approximation as well as a bootstrap resampling scheme for consistent estimation of quantiles of the quantity in~\eqref{eq:MaximumDiffIneqCor}.

It is easy to prove a result similar to Corollary~\ref{cor:MarginalScreening} for marginal screening in Cox proportional hazards model. 
\subsection{Post-selection Inference under Covariate Selection}\label{sec:PoSI}
In the previous section, we have considered asymptotic linear representation uniform over all models of size $1$. In this section, we consider linear representation error uniform over all models of size bounded by $k\,(\ge 1)$. This is important for post-selection inference (PoSI). In the context of regression analysis, the PoSI problem refers to the construction of confidence regions for $\beta_{n,\hat{M}}$ for a model $\hat{M}\subseteq\{1,2,\ldots,p\}$ chosen based on the data $(X_1, Y_1), \ldots, (X_n, Y_n)\in\mathbb{R}^p\times\mathbb{R}$. Formally, for any $M\subseteq\{1,2,\ldots,p\}$, define the estimator
\[
\hat{\beta}_{n,M} := \argmin_{\theta\in\mathbb{R}^{|M|}}\,\frac{1}{n}\sum_{i=1}^n h(X_i(M))\ell(\theta^{\top}X_i(M), Y_i),
\]
for some twice differentiable convex loss function $\ell(\cdot, \cdot)$ and non-negative weight function $h(\cdot)$. Based on the results in previous sections, we can consider the target parameters
\[
\beta_{n,M} := \argmin_{\theta\in\mathbb{R}^{|M|}}\,\frac{1}{n}\sum_{i=1}^n \mathbb{E}\left[h(X_i(M))\ell(\theta^{\top}X_i(M), Y_i)\right].
\]
Let $\mathcal{M}$ be a collection of subsets of $\{1,2,\ldots, p\}$. The PoSI problem for the collection of targets $\{\beta_{n,M}:\,M\in\mathcal{M}\}$ concerns the construction of a collection of confidence regions $\{\hat{\mathcal{R}}_{n,M}:\,M\in\mathcal{M}\}$ of level $\alpha$ satisfying
\begin{equation}\label{eq:PoSIGuarantee}
\liminf_{n\to\infty}\,\mathbb{P}\left(\beta_{n,\hat{M}}\in \hat{\mathcal{R}}_{n,\hat{M}}\right) \ge 1 - \alpha,
\end{equation}
for any model $\hat{M}$ chosen possibly depending on the data $\{(X_i, Y_i)\}_{1\le i\le n}$ that satisfies $\mathbb{P}(\hat{M}\in\mathcal{M}) = 1$; see~\cite{kuchibhotla2018valid} for more details. Theorem 3.1 of \cite{kuchibhotla2018valid} proves that the post-selection inference guarantee~\eqref{eq:PoSIGuarantee} is equivalent to the simultaneous guarantee:
\[
\liminf_{n\to\infty}\,\mathbb{P}\left(\bigcap_{M\in\mathcal{M}}\left\{\beta_{n,M}\in \hat{\mathcal{R}}_{n,M}\right\}\right) \ge 1 - \alpha.
\]
It is easy to see that for a post-selection confidence region $\hat{\mathcal{R}}_{n,\hat{M}}$ based on $\hat{\beta}_{n,\hat{M}}$ to have a Lebesgue measure (on $\mathbb{R}^{|M|}$) converging to zero, it is necessary that
\[
\sup_{M\in\mathcal{M}}\,\norm{\hat{\beta}_{n,M} - \beta_{n,M}} = o_p(1),\quad\mbox{as}\quad n\to\infty,
\] 
for some norm $\norm{\cdot}$. Based on our deterministic inequalities in previous sections, we can provide precise statements of uniform convergence. We provide only one such result similar to Corollary~\ref{cor:MarginalScreening}. To state the results, define for $M\subseteq\{1,2,\ldots,p\}$ and $\theta\in\mathbb{R}^{|M|},$
\[
\hat{L}_{n,M}(\theta) := \frac{1}{n}\sum_{i=1}^n h(X_i(M))\ell(\theta^{\top}X_i(M), Y_i).
\]
Also, set
\[
\delta_{n,M} := 1.5\norm{[\nabla_2\hat{L}_n(\beta_{n,M})]^{-1}\nabla\hat{L}_n(\beta_{n,M})}_2.
\]
\begin{cor}\label{cor:PoSIMEst}
Suppose 
\[
\max_{1\le i\le n}\,\max_{M\in\mathcal{M}}\,C\left(\norm{X_i(M)}_2\delta_{n,M}, Y_i\right) \le \frac{4}{3},
\]
then for each $M\in\mathcal{M}$, there exists a unique vector $\hat{\beta}_{n,M}\in\mathbb{R}^{|M|}$ satisfying $\nabla \hat{L}_n(\hat{\beta}_{n,M}) = 0$ and
\begin{align*}
&\norm{\hat{\beta}_{n,M} - \beta_{n,M} + [\nabla_2\hat{L}_n(\beta_{n,M})]^{-1}\nabla\hat{L}_n(\beta_{n,M})}_2 \le \left[\max_{1\le i\le n}\,C(\norm{X_i(M)}_2\delta_{n,M}) - 1\right]\delta_{n,M}.
\end{align*}
\end{cor}
\begin{proof}
The proof follows trivially from Corollary~\ref{cor:RegressionEst}.
\end{proof}
As in Section~\ref{subsec:MarginalScreen}, the linear expansion result of Corollary~\ref{cor:PoSIMEst} above proves that 
\[
\hat{\beta}_{n,M} - \beta_{n,M} ~=~ (1 + o_p(1))[\nabla_2 \hat{L}_n(\beta_{n,M})]^{-1}\nabla \hat{L}_n(\beta_{n,M})\quad\mbox{uniformly for}\quad M\in\mathcal{M}.
\]
Therefore, one can apply various bootstrap schemes to evaluate quantiles or approximate the distribution of $\{\hat{\beta}_{n,M} - \beta_{n,M}:\,M\in\mathcal{M}\}$ under various dependence settings. For simplicity and concreteness, we have dealt with covariate selection here and using techniques from previous section, it is not difficult to also consider post-selection inference problems related to family of transformations on the covariates/response.
\section{Conclusions and Future Work}\label{sec:Conclude}
In this work, we have provided deterministic inequalities for a class of smooth $M$-estimators that unify the classical asymptotic analysis under various dependence settings. Furthermore, these inequalities readily yield tail bounds for estimation errors as well as asymptotic expansions. A connection between these deterministic inequalities and semilocal convergence analysis of iterative algorithms is established. 

Throughout the paper we have considered only twice differentiable loss functions. It is of interest to understand the non-smooth loss functions like the absolute deviation, Huber's loss from the viewpoint of deterministic inequalities. As described in Section~\ref{sec:Banach}, several iterative algorithms exist with linear/superlinear convergence also for non-smooth functions. We hope to present similar deterministic inequalities for non-smooth $M$-estimators in the future.  
\section*{Acknowledgments}
The author would like to thank Mateo Wirth and Bikram Karmakar for helpful discussions and suggestions.
\bibliographystyle{apalike}
\bibliography{../AssumpLean}
\appendix
\section{Proofs of Results in Section~\ref{sec:Banach}}\label{appsec:NonSingular}
\subsection{Proof of Theorem~\ref{thm:NonSingular}}
\begin{proof}
The proof essentially from the arguments of \cite{Yuan98} but it was stated there with the hypothesis of continuous differentiability of $f(\cdot)$. Only everywhere differentiability of $f(\cdot)$ is required. Define
\[
\varphi(\theta) := \theta - A^{-1}f(\theta).
\]
To finish the proof it is enough to show that there exists a fixed point for $\varphi(\cdot)$ in $B(\theta_0, r)$. Let $I$ represent the identity matrix in $\mathbb{R}^q$. Since
\[
\nabla \varphi(\theta) = I - A^{-1}\nabla f(\theta) = A^{-1}\left(A - \nabla f(\theta)\right),
\] 
and for all $\theta\in B(\theta_0, r)$, $\norm{\nabla \varphi(\theta)}_{op} \le \varepsilon$ by~\eqref{assump:Contraction}. This implies that $\varphi(\cdot)$ is a contraction mapping on $B(\theta_0, r)$ with contraction constant $\varepsilon$. Also, since~\eqref{assump:FixedPoint} implies
\[
\norm{\varphi(\theta_0) - \theta_0}_2 = \norm{A^{-1}f(\theta_0)}_2 \le r(1 - \varepsilon),
\]
it follows that for $\theta\in B(\theta_0, r)$,
\[
\norm{\varphi(\theta) - \theta_0}_2 \le \norm{\varphi(\theta) - \varphi(\theta_0)}_2 + \norm{\varphi(\theta_0) - \theta_0}_2 \le \varepsilon\norm{\theta - \theta_0}_2 + r(1 - \varepsilon) \le r.
\]
Thus, $\varphi:B(\theta_0, r)\to B(\theta_0, r)$ is a contraction and hence has a unique fixed point in $B(\theta_0, r)$ by the fixed point theorem. See \citet[Theorem 9.1]{LOOM68} for more details on contraction mapping fixed point theorem.

Now observe that by a first order Taylor series expansion
\[
0 = f(\theta^{\star}) = f(\theta_0) + \nabla f(\bar{\theta})\left(\theta - \theta_0\right),
\]
for some $\bar{\theta}$ that lies on the line segment joining $\theta^{\star}$ and $\theta_0$. Thus,
\begin{equation}\label{eq:InterimLinearity}
-A^{-1}f(\theta_0) = A^{-1}\nabla f(\bar{\theta})\left(\theta^{\star} - \theta_0\right).
\end{equation}
Since ${\theta}^{\star}\in B(\theta_0, r)$, it follows that $\bar{\theta}\in B(\theta_0, r)$ and so, by~\eqref{assump:Contraction},
\[
\norm{A^{-1}(A - \nabla f(\bar{\theta}))}_{op} \le \varepsilon\quad\Rightarrow\quad (1 - \varepsilon)I \preceq A^{-1}\nabla f(\bar{\theta}) \preceq (1 + \varepsilon)I. 
\]
Therefore, $A^{-1}\nabla f(\bar{\theta})$ is invertible and~\eqref{eq:InterimLinearity} leads to,
\[
\norm{\theta^{\star} - \theta_0}_2 = \norm{\left(A^{-1}\nabla f(\bar{\theta})\right)^{-1}A^{-1}f(\theta_0)}_2,
\]
and
\[
\frac{1}{1 + \varepsilon}\norm{A^{-1}f(\theta_0)}_2 \le \norm{\theta^{\star} - \theta_0}_2 \le \frac{1}{1 - \varepsilon}\norm{A^{-1}f(\theta_0)}_2.
\]

\end{proof}
\subsection{Proof of Theorem~\ref{thm:NewtonStepDifferentiable}}
\begin{proof}
Define
\[
\varepsilon := 1/3,\quad\mbox{and}\quad  r := 1.5\norm{(\nabla f(\theta_0))^{-1}f(\theta_0)}_2.
\]
From these definitions, it is clear that
\[
\norm{(\nabla f(\theta_0))^{-1}f(\theta_0)}_2 = r(1 - \varepsilon),
\]
and
\begin{align*}
\norm{(\nabla f(\theta_0))^{-1}(\nabla f(\theta_0) - \nabla f(\theta))}_{op} &\le L\norm{\theta - \theta_0}_2^{\alpha}\\ 
&\le Lr^{\alpha} \le \frac{L}{(1 - \varepsilon)^{\alpha}}\norm{(\nabla f(\theta_0))^{-1}f(\theta_0)}_2^{\alpha} \le 1/3,
\end{align*}
under the assumption~\eqref{eq:AssumptionNewton}. Therefore, the conditions of Theorem~\ref{thm:NonSingular} are satisfied and so, there exists a unique solution $\theta^{\star}\in B(\theta_0, r)$ satisfying $f(\theta^{\star}) = 0$. Also, it follows that
\[
\norm{\theta^{\star} - \theta_0}_2 \le 1.5\norm{(\nabla f(\theta_0))^{-1}f(\theta_0)}_2.
\]
Observe now that
\begin{align*}
\norm{\theta_0 - (\nabla f(\theta_0))^{-1}f(\theta_0) - {\theta}^{\star}}_2 &= \norm{(\nabla f(\theta_0))^{-1}\left( - f(\theta_0) - [\nabla f(\theta_0)](\theta^{\star} - \theta_0)\right)}_2\\
&= \norm{(\nabla f(\theta_0))^{-1}\left(f(\theta^{\star}) - f(\theta_0) - [\nabla f(\theta_0)](\theta^{\star} - \theta_0)\right)}_2\\
&\overset{(a)}{=} \norm{(\nabla f(\theta_0))^{-1}\left(\nabla f(\bar{\theta}) - \nabla f(\theta_0)\right)(\theta^{\star} - \theta_0)}_2\\
&\le \norm{(\nabla f(\theta_0))^{-1}\left(\nabla f(\bar{\theta}) - \nabla f(\theta_0)\right)}_{op}\norm{\theta_0 - \theta^{\star}}_2\\
&\overset{(b)}{\le} L\norm{\theta_0 - \theta^{\star}}^{1 + \alpha}\\
&\le (1.5)^{1 + \alpha}L\norm{(\nabla f(\theta_0))^{-1}f(\theta_0)}_2^{1 + \alpha}.
\end{align*}
Equality~(a) above follows from the mean-value theorem for some vector $\bar{\theta}$ that lies on the line segment joining $\theta^{\star}, \theta_0$ and inequality~(b) follows from the fact $\norm{\bar{\theta} - \theta_0}_2 \le \norm{\theta_0 - \theta^{\star}}_2.$
\end{proof}
\section{Proofs of Results in Section~\ref{sec:GLM}}\label{appsec:GLM}
\subsection{Proof of Theorem~\ref{thm:ConvexMEst}}
\begin{proof}
To prove~\eqref{eq:Part1ConvexMEst}, we verify the assumptions of Theorem~\ref{thm:NonSingular}. Take in Theorem~\ref{thm:NonSingular},
\[
f(\theta) := [\hat{\mathcal{Q}}_n(\theta_0)]^{-1}\hat{\mathcal{Z}}_n(\theta),\quad A := I,\quad\mbox{and}\quad r = \delta_n(\theta_0),\quad \varepsilon = \frac{1}{3}.
\]
Here $I$ represents the identity matrix in $\mathbb{R}^q$. Condition~\eqref{assump:FixedPoint} is trivially satisfied since
\begin{equation}\label{eq:verifyAssFixedPoint}
\norm{A^{-1}f(\theta_0)}_2 = \norm{[\hat{\mathcal{Q}}_n(\theta_0)]^{-1}\hat{\mathcal{Z}}_n(\theta_0)}_2 = \frac{2\delta_n(\theta_0)}{3} = (1 - \varepsilon)\delta_{n}(\theta_0). 
\end{equation}
To verify condition~\eqref{assump:Contraction}, note that
\begin{align*}
\norm{A^{-1}\left(A - \nabla f(\theta)\right)}_{op} &= \norm{[\nabla \hat{\mathcal{Z}}_n(\theta_0)]^{-1}\left(\nabla \hat{\mathcal{Z}}_n(\theta_0) - \nabla\hat{\mathcal{Z}}_n(\theta)\right)}_{op}\\
&= \sup_{e\in\mathbb{R}^q:\,\norm{e}_2 = 1}\left|\frac{e^{\top}\nabla \hat{\mathcal{Z}}_n(\theta)e}{e^{\top}\nabla \hat{\mathcal{Z}}_n(\theta_0)e} - 1\right|.
\end{align*}
To control the right hand side above, note that by the definition of $C(u, w)$,
\begin{equation}\label{eq:BothSides}
\begin{split}
e^{\top}\nabla \hat{\mathcal{Z}}_n(\theta)e &= \frac{1}{n}\sum_{i=1}^n e^{\top}\nabla L(\theta, W_i)e \le \frac{1}{n}\sum_{i=1}^n \left\{e^{\top}\nabla L(\theta_0, W_i)e\right\}C(r, W_i),\\
e^{\top}\nabla \hat{\mathcal{Z}}_n(\theta_0)e &= \frac{1}{n}\sum_{i=1}^n e^{\top}\nabla L(\theta_0, W_i)e \le \frac{1}{n}\sum_{i=1}^n \left\{e^{\top}\nabla L(\theta, W_i)e\right\}C(r, W_i).
\end{split}
\end{equation}
Thus under Assumption~\ref{eq:Event}, for all $e\in\mathbb{R}^q$ with $\norm{e}_2 = 1$ and $\theta\in\mathbb{R}^q$ such that $\norm{\theta - \theta_0}_2 \le r$,
\begin{equation}\label{eq:RatioBounded}
\frac{3}{4} \le \frac{e^{\top}\hat{\mathcal{Q}}_n(\theta)e}{e^{\top}\hat{\mathcal{Q}}_n(\theta_0)e} \le \frac{4}{3},
\end{equation}
and so,
\begin{equation}\label{eq:VerifyContraction}
\sup_{\theta\in B_{r}(\theta_0)}\norm{A^{-1}(A - \nabla f(\theta_0))}_{op} \le \max\left\{\frac{1}{3}, \frac{1}{4}\right\} = \frac{1}{3} = \varepsilon.
\end{equation}
Inequalities~\eqref{eq:verifyAssFixedPoint} and~\eqref{eq:VerifyContraction} complete the verification of condition~\eqref{assump:FixedPoint} and~\eqref{assump:Contraction}, respectively with $\varepsilon = 1/3$. Therefore, by Theorem~\ref{thm:NonSingular}, we get that there exists $\hat{\theta}_n\in\mathbb{R}^q$ satisfying
\[
\hat{\mathcal{Z}}_n(\hat{\theta}_n) = 0,\quad\mbox{and}\quad \frac{1}{2}\delta_n(\theta_0) \le \norm{\hat{\theta}_n - \theta_0}_2 \le \delta_n(\theta_0).
\]
Thus, the first part of the result is proved. 

To prove~\eqref{eq:Part2ConvexMEst}, note by a Taylor series expansion of $\hat{\mathcal{Z}}_n(\hat{\theta}_n)$ around $\theta_0$ that,
\[
0 = \hat{\mathcal{Z}}_n(\hat{\theta}_n) = \hat{\mathcal{Z}}_n(\theta_0) + \hat{\mathcal{Q}}_n(\bar{\theta})\left(\hat{\theta}_n - \theta_0\right),
\]
for some $\bar{\theta}$ that lies on the line segment joining $\hat{\theta}_n$ and $\theta_0$. Multiplying both sides by $\hat{\mathcal{Q}}_n(\theta_0)$, we get
\begin{equation}\label{eq:IntrimLinearRep}
-[\hat{\mathcal{Q}}_n(\theta_0)]^{-1}\hat{\mathcal{Z}}_n(\theta_0) = [\hat{\mathcal{Q}}_n(\theta_0)]^{-1}\hat{\mathcal{Q}}_n(\bar{\theta})\left(\hat{\theta} - \theta_0\right).
\end{equation}
By~\eqref{eq:BothSides}, it follows that
\[
\frac{1}{\max_{i}\, C(\delta_n(\theta_0), W_i)}\hat{\mathcal{Q}}_n(\theta_0) ~\preceq~ \hat{\mathcal{Q}}_n(\bar{\theta}) ~\preceq~ \max_i\,C(\delta_n(\theta_0), W_i)\hat{\mathcal{Q}}_n(\theta_0),
\]
which implies that
\begin{align*}
\norm{[\hat{\mathcal{Q}}_n(\theta_0)]^{-1}\hat{\mathcal{Q}}_n(\bar{\theta}) - I}_{op} &\le \left(\max_{1\le i\le n}C(\delta_{n}(\theta_0), W_i) - 1\right)\max\left\{1,\,\frac{1}{\max_i\,C(\delta_n(\theta_0), W_i)}\right\}\\
&\le \max_{1\le i\le n}C(\delta_{n}(\theta_0), W_i) - 1,
\end{align*}
since $C(r, w) \ge 1$ for all $r$ and $w$. Therefore, using~\eqref{eq:IntrimLinearRep}, we obtain
\begin{align*}
\norm{\hat{\theta}_n - \theta_0 + [\hat{\mathcal{Q}}_n(\theta_0)]^{-1}\hat{\mathcal{Z}}_n(\theta_0)}_2 &\le \norm{[\hat{\mathcal{Q}}_n(\theta_0)]^{-1}\hat{\mathcal{Q}}_n(\bar{\theta}) - I}_{op}\norm{\hat{\theta}_n - \theta_0}_2\\
&\le \delta_n(\theta_0)\left(\max_{1\le i\le n}C(\delta_{n}(\theta_0), W_i) - 1\right).
\end{align*}
This completes the proof.
\end{proof}
\subsection{Proof of Corollary~\ref{cor:RegressionEst}}
\begin{proof}
Take $w = (x, y)$ and $L(\theta; w) = h(x)\ell(x^{\top}\theta, y)$ in Theorem~\ref{thm:ConvexMEst}. For this function, 
\[
\nabla_2L(\theta, w) = h(x)\ell''(x^{\top}\theta, y)xx^{\top}.
\]
To verify assumption~\ref{eq:Event}, note that
\begin{align*}
\sup_{\norm{\theta_1 - \theta_2}_2 \le u}\sup_{e\in\mathbb{R}^p:\norm{e}_2 = 1}\frac{e^{\top}\nabla_2L(\theta_1, w)e}{e^{\top}\nabla_2L(\theta_2, w)e} &= \sup_{\norm{\theta_1 - \theta_2}_2 \le u}\,\frac{\ell''(x^{\top}\theta_1, y)}{\ell''(x^{\top}\theta_2, y)}\\
&\le \sup_{|s - t| \le \norm{x}_2u}\,\frac{\ell''(s, y)}{\ell''(t, y)} = C\left(\norm{x}_2u, y\right).
\end{align*}
Therefore, under~\eqref{eq:AssumptionA2}, assumption~\ref{eq:Event} holds true and the result follows.
\end{proof}
\subsection{Proof of Proposition~\ref{prop:Convexity}}
\begin{proof}
For any four real non-negative numbers $a, b, c$ and $d$,
\begin{equation}\label{eq:RatioInequality}
\min\left\{\frac{a}{b}, \frac{c}{d}\right\} \le \frac{a + c}{b + d} \le \max\left\{\frac{a}{b}, \frac{c}{d}\right\}.
\end{equation}
Suppose $L_1(\cdot, \cdot)$ and $L_2(\cdot, \cdot)$ be any two elements of $\mathcal{C}_T$. Fix two positive real numbers $\alpha, \beta$ and set $L(\theta, w) = \alpha L_1(\theta, w) + \beta L_2(\theta, w)$. It follows that $L(\cdot, \cdot)$ is convex in the first argument and
\[
\nabla_2L(\theta, w) = \alpha \nabla_2L_1(\theta, w) + \beta \nabla_2L_2(\theta, w).
\]
Fix $u \ge 0$. Then for each $\theta_1, \theta_2$ satisfying $\norm{\theta_1 - \theta_2}_2 \le u$, and $e$ satisfying $\norm{e}_2 = 1$,
\[
\frac{e^{\top}\nabla_2L(\theta_1, w)e}{e^{\top}\nabla_2L(\theta_2, w)e} \le \max\left\{\frac{e^{\top}\nabla_2L_1(\theta_1, w)e}{e^{\top}\nabla_2L_1(\theta_2, w)e}, \frac{e^{\top}\nabla_2L_2(\theta_1, w)e}{e^{\top}\nabla_2L_2(\theta_2, w)e}\right\} \le T(u, w),
\]
by inequality \eqref{eq:RatioInequality}. Therefore, $L(\cdot, \cdot)\in\mathcal{C}_T$. Note that $\mathcal{C}_T$ is a non-empty set since the any function whose second derivative is a non-negative multiple of $T(\cdot, \cdot)$ belongs to $\mathcal{C}_T$.
\end{proof}
\section{Proofs of Results in Section~\ref{sec:Cox}}\label{appsec:Cox}
\subsection{A Preliminary Lemma}
We need to following lemma for the proof of Theorem~\ref{thm:CoxModel}. The result is similar to Lemma A2 of~\cite{HjortPollard}.
\begin{lem}\label{lem:SoftMax}
Suppose $K(t) := \log R(t)$, where 
\[
R(t) := \sum_{i=1}^n w_i\exp(a_i t)\quad\mbox{for}\quad w_i\ge 0,\, a_i\in\mathbb{R}.
\]
Assume that not all $w_i$'s are zero. Then $K(t)$ is convex with derivatives
\[
K'(t) = \sum_{i=1}^n a_iv_i(t) =: \bar{a}(t),\quad\mbox{and}\quad K''(t) = \sum_{i = 1}^n v_i(t)\left\{a_i - \bar{a}(t)\right\}^2,
\]
where $v_i(t) := {w_i\exp(a_it)}/{R(t)}$ for $1\le i\le n.$ Furthermore, for $t\in\mathbb{R}$ and for all $0 \le |s| \le |t|$,
\[
\max\left\{\left|\frac{K''(s)}{K''(0)} - 1\right|, \left|\frac{K''(0)}{K''(s)} - 1\right|\right\} \le 4\mu_n|t|\exp(4\mu_n|t|),
\]
where $\mu_n := \max_{1\le i\le n}\left|a_i - \bar{a}(0)\right|.$
\end{lem}
\begin{proof}
It is easy to verify that
\[
K'(t) = \frac{R'(t)}{R(t)} = \frac{\sum_{i=1}^n w_ia_i\exp(a_it)}{R(t)} = \sum_{i=1}^n a_iv_i(t).
\]
Thus,
\[
K''(t) = \frac{R''(t)}{R(t)} - \left(\bar{a}(t)\right)^2 = \sum_{i=1}^n a_i^2v_i(t) - \left(\sum_{i=1}^n a_iv_i(t)\right)^2 = \sum_{i=1}^n v_i(t)\left\{a_i - \bar{a}(t)\right\}^2.
\]
Since $K''(t) \ge 0$ for all $t\ge 0$, $K(\cdot)$ is a convex function.

To prove the second part, fix $s$ satisfying $|s| \le |t|$. Clearly,
\[
v_i(s) = \frac{w_i\exp(a_is)}{R(s)} = \frac{w_i}{R(0)}\left[\frac{\exp(a_is)\sum_{j = 1}^n w_j}{\sum_{j = 1}^n w_j\exp(a_js)}\right] = v_i(0)(1 + \varepsilon_i(s)),
\]
where
\[
1 + \varepsilon_i(s) := \frac{\exp(a_is)\sum_{j = 1}^n w_j}{\sum_{j = 1}^n w_j\exp(a_js)} = \frac{\exp(\{a_i - \bar{a}(0)\}s)\sum_{j = 1}^n w_j}{\sum_{j = 1}^n w_j\exp(\{a_j -\bar{a}(0)\}s)}.
\]
It is easy to check that
\[
\min_{1\le j\le n}\,\exp\left(\{a_j - \bar{a}(0)\}s\right) \le \frac{\sum_{j = 1}^n w_j\exp(\{a_j - \bar{a}(0)\}s)}{\sum_{j = 1}^n w_j} \le \max_{1\le j\le n}\,\exp\left(\{a_j - \bar{a}(0)\}s\right).
\]
Therefore, for all $|s| \le |t|$ and $1\le i\le n$,
\begin{equation}\label{eq:EpsilonIneq}
\exp\left(-\mu_n|t|\right) \le 1 + \varepsilon_i(s) \le \exp\left(\mu_n|t|\right).
\end{equation}
This implies that
\begin{equation}\label{eq:v_is-0inequality}
\max\left\{\frac{v_i(s)}{v_i(0)}, \frac{v_i(0)}{v_i(s)}\right\} \le \exp(\mu_n|t|).
\end{equation}
Observe that
\begin{equation*}
\begin{split}
K''(s) &= \sum_{i=1}^n v_i(s)\left(a_i - \bar{a}(s)\right)^2\\
&= \sum_{i=1}^n v_i(0)(1 + \varepsilon_i(s))\left(a_i - \bar{a}(0) + \bar{a}(0) - \bar{a}(s)\right)^2\\
&= \sum_{i=1}^n v_i(0)(a_i - \bar{a}(0))^2(1 + \varepsilon_i(s)) + \sum_{i=1}^n v_i(0)(\bar{a}(0) - \bar{a}(s))^2(1 + \varepsilon_i(s))\\
&\qquad + 2\sum_{i=1}^n v_i(0)(a_i - \bar{a}(0))(\bar{a}(0) - \bar{a}(s))(1 + \varepsilon_i(s)).
\end{split}
\end{equation*}
We now subtract $K''(0)$ and bound the remainder.
\begin{equation}\label{eq:FirstPart}
\begin{split}
|K''(s) &- K''(0)|\\ &\le \sum_{i=1}^n v_i(0)(a_i - \bar{a}(0))^2\varepsilon_i(s) + \left|\bar{a}(0) - \bar{a}(s)\right|^2\sum_{i=1}^n v_i(0)(1 + \varepsilon_i(s))\\
&\quad + 2\left|\bar{a}(0) - \bar{a}(s)\right|\times\left|\sum_{i=1}^n v_i(0)(a_i - \bar{a}(0))(1 + \varepsilon_i(s))\right|\\
&\le K''(0)\left(\exp(\mu_n|t|) - 1\right) + \left|\bar{a}(0) - \bar{a}(s)\right|^2\max_{1\le i\le n}(1 + \varepsilon_i(s))\\
&\quad + 2\left|\bar{a}(0) - \bar{a}(s)\right|\left(\sum_{i=1}^n v_i(0)(a_i - \bar{a}(0))^2\right)^{1/2}\left(\sum_{i=1}^n v_i(0)(1 + \varepsilon_i(s))^2\right)^{1/2}\\
&\le K''(0)\left(\exp(\mu_n|t|) - 1\right)  + \left|\bar{a}(0) - \bar{a}(s)\right|^2\max_{1\le i\le n}(1 + \varepsilon_i(s))\\
&\quad + 2\left|\bar{a}(0) - \bar{a}(s)\right|\left(K''(0)\right)^{1/2}\max_{1\le i\le n}(1 + \varepsilon_i(s))\\
&\le K''(0)\left(\exp(\mu_n|t|) - 1\right) + \left|\bar{a}(0) - \bar{a}(s)\right|^2\exp(2\mu_n|t|)\\
&\quad + 2\left|\bar{a}(0) - \bar{a}(s)\right|\left(K''(0)\right)^{1/2}\exp(\mu_n|t|).
\end{split}
\end{equation}
Here the last inequality follows from inequality~\eqref{eq:EpsilonIneq}. To bound $|\bar{a}(0) - \bar{a}(s)|$, note that
\begin{align*}
|\bar{a}(0) - \bar{a}(s)| &= \left|\sum_{i=1}^n v_i(s)(a_i - \bar{a}(0))\right|\\
&= \left|\sum_{i=1}^n v_i(0)(a_i - \bar{a}(0))(1 + \varepsilon_i(s))\right|\\
&\overset{(a)}{=} \left|\sum_{i=1}^n v_i(0)(a_i - \bar{a}(0))\varepsilon_i(s)\right|\\
&\le \left(\sum_{i=1}^n v_i(0)\varepsilon_i^2(s)\right)^{1/2}\left(\sum_{i=1}^n v_i(0)(a_i - \bar{a}(0))^2\right)^{1/2}\\
&\le (K''(0))^{1/2}(\exp(\mu_n|t|) - 1),
\end{align*}
where the equality (a) follows from the fact that $\bar{a}(0) = \sum v_i(0)a_i$ and $\sum v_i(0) = 1$. Substituting this inequality in~\eqref{eq:FirstPart}, we get
\begin{align*}
|K''(s) - K''(0)| &\le K''(0)(\exp(\mu_n|t|) - 1) + K''(0)\exp(2\mu_n|t|)(\exp(\mu_n|t|) - 1)^2\\ &\qquad+ 2K''(0)\exp(\mu_n|t|)(\exp(\mu_n|t|) - 1)\\
&= K''(0)(\exp(\mu_n|t|) - 1)\left[1 + \exp(2\mu_n|t|)(\exp(\mu_n|t|) - 1) + 2\exp(\mu_n|t|)\right]\\
&\le K''(0)(\exp(\mu_n|t|) - 1)\left[1 + \exp(3\mu_n|t|) + 2\exp(3\mu_n|t|)\right]\\
&\le 4K''(0)(\exp(\mu_n|t|) - 1)\exp(3\mu_n|t|)\\
&\le 4K''(0)\mu_n|t|\exp(4\mu_n|t|).
\end{align*}
Therefore, for all $|s| \le |t|,$
\[
\left|\frac{K''(s)}{K''(0)} - 1\right| \le 4\mu_n|t|\exp(4\mu_n|t|).
\]
The bound for $K''(0)/K''(s)$ follows the same line of argument as~\eqref{eq:FirstPart} and finally use inequality~\eqref{eq:v_is-0inequality}.
\end{proof}
\subsection{Proof of Theorem~\ref{thm:CoxModel}}
\begin{proof}
To prove~\eqref{eq:CoxAsympLinear}, we verify the assumptions of Theorem~\ref{thm:NonSingular} with
\[
f(\beta) := [\nabla \hat{\mathcal{Z}}_{n}(\beta_{0})]^{-1}\hat{\mathcal{Z}}_{n}(\beta),\quad A := I\quad\mbox{and}\quad r := \delta_{n}(\beta_0),\quad \varepsilon = 1/3.
\]
Assumption~\eqref{assump:FixedPoint} is trivially satisfied by the definition of $r$ and to verify Assumption~\eqref{assump:Contraction}, it is enough to verify for all $\nu\in\mathbb{R}^{p}$ with $\norm{\nu}_2 \le r$, that
\begin{equation}\label{eq:ToVerifyContraction}
\norm{[\nabla \hat{\mathcal{Z}}_{n}(\beta_{0})]^{-1}\left(\nabla \hat{\mathcal{Z}}_{n}(\beta_{0}) - \nabla \hat{\mathcal{Z}}_{n}(\beta_{0} + \nu)\right)}_{op} \le 1/3.
\end{equation}
For any fixed $0 \le s < \infty, \nu\in\mathbb{R}^{p}$, define
\[
K(\ell) := \log\left(\sum_{i=1}^n w_i\exp\left(a_i\ell\right)\right),
\]
where
\[
w_i := H_{2}(X_{i,s})Y_i(s)\exp\left(\beta_{0}^{\top}X_{i,s}\right)\quad\mbox{and}\quad a_i = \nu^{\top}X_{i,s}.
\]
Then $K(\ell) = \log R_n(s, \beta_{0} + \ell\nu)$. As in Lemma~\ref{lem:SoftMax}, set
\begin{equation}\label{eq:DefMu}
\tilde{\mu}_n := \max_{1\le i\le n}|a_i - \bar{a}(0)| = \max_{1\le i\le n}\left|\nu^{\top}\left(X_{i,s} - \bar{X}_{n,s}(\beta_0)\right)\right|.
\end{equation}
It is evident that
\begin{align*}
\hat{\mathcal{Z}}_{n}(\beta_{0} + \ell\nu) &= \sum_{i=1}^n \int_0^{\infty} H_{1}(X_{i,s})\left\{K'(\ell) - X_{i,s}\right\}dN_{i}(s),\\
\frac{d}{d\beta} \hat{\mathcal{Z}}_{n}(\beta)\bigg|_{\beta = \beta_{0} + \ell\nu} &= \sum_{i=1}^n \int_0^{\infty} H_{1}(X_{i,s})\left\{K''(\ell)\right\}dN_i(s),\\
\frac{d}{d\beta} \hat{\mathcal{Z}}_{n}(\beta)\bigg|_{\beta = \beta_{0}} &= \sum_{i=1}^n \int_0^{\infty} H_{1}(X_{i,s})\left\{K''(0)\right\}dN_i(s).
\end{align*}
The dependence of $K(\cdot)$ on $s$ is suppressed in the formulas above. From Lemma~\ref{lem:SoftMax}, we have for all $0 \le \ell \le 1$,
\begin{equation}\label{eq:ApplicationSoftMax}
K''(0)\left[1 - 4\tilde{\mu}_n\exp(4\tilde{\mu}_n)\right] \le K''(\ell) \le K''(0)\left[1 + 4\tilde{\mu}_n\exp(4\tilde{\mu}_n)\right]
\end{equation}
Clearly from the definition~\eqref{eq:DefMu},
\begin{equation*}
\tilde{\mu}_n \le \mu_{n}(s)\norm{\nu}_2 \le \mu_{n}(s)r = \mu_{n}(s)\delta_{n}(\beta_0) \le 1/16.
\end{equation*}
Hence, $4\tilde{\mu}_n \le 1/4$ and so, $4\tilde{\mu}_n\exp(4\tilde{\mu}_n) \le 1/3$. Substituting this inequality in~\eqref{eq:ApplicationSoftMax}, we get
\[
\frac{2}{3}K''(0) \le K''(1) \le \frac{4}{3}K''(0),
\]
and so,
\[
\frac{2}{3}\nabla \hat{\mathcal{Z}}_{n}(\beta_{0}) \preceq \nabla \hat{\mathcal{Z}}_{n}(\beta_{0} + \nu) \preceq \frac{4}{3}\nabla \hat{\mathcal{Z}}_{n}(\beta_{0}),
\]
proving~\eqref{eq:ToVerifyContraction} for all $\norm{\nu}_2 \le r$ and $M\in\mathcal{M}$. Hence from Theorem~\ref{thm:NonSingular}, we get that there exists a solution $\hat{\beta}_{n}$ such that 
\[
\hat{\mathcal{Z}}_{n}(\hat{\beta}_{n}) = 0\quad\mbox{and}\quad \frac{\delta_n(\beta_0)}{2} \le \norm{\hat{\beta}_{n} - \beta_{0}}_2 \le 2\delta_{n}(\beta_0).
\]

To prove the linear representation part of the result, we follow the proof of Theorem~\ref{thm:ConvexMEst}. By a Taylor series expansion, we get that
\begin{equation}\label{eq:CoxTaylor}
0 = \hat{\mathcal{Z}}_{n}(\hat{\beta}_{n}) = \hat{\mathcal{Z}}_{n}(\beta_{0}) + \nabla\hat{\mathcal{Z}}_{n}(\bar{\beta})\left(\hat{\beta}_{n} - \beta_{0}\right),
\end{equation}
for some vector $\bar{\beta}$ that lies on the line segment between $\beta_{0}$ and $\hat{\beta}_{n}$. This implies that
\[
-[\nabla\hat{\mathcal{Z}}_n(\beta_0)]^{-1}\hat{\mathcal{Z}}_n(\beta_0) = [\nabla\hat{\mathcal{Z}}_n(\beta_0)]^{-1}\nabla\hat{\mathcal{Z}}_n(\bar{\beta})(\hat{\beta}_n - \beta_0).
\]
From~\eqref{eq:ApplicationSoftMax}, it follows that
\[
\left(1 - \gamma_n\right)I \preceq [\nabla\hat{\mathcal{Z}}_n(\beta_0)]^{-1}\nabla\hat{\mathcal{Z}}_n(\bar{\beta}) \preceq \left(1 + \gamma_n\right)I,
\]
where
\[
\gamma_n = 4\sup_{0 \le s < \infty}\mu_n(s)\delta_n(\beta_0)\exp\left(4\sup_{0 \le s < \infty}\mu_n(s)\delta_n(\beta_0)\right) \le 4e^{1/4}\sup_{0 \le s < \infty}\mu_n(s)\delta_n(\beta_0).
\]
Therefore,
\[
\norm{\hat{\beta} - \beta_0 + [\hat{\mathcal{J}}_n(\beta_0)]^{-1}\hat{\mathcal{Z}}_n(\beta_0)}_2 \le 8e^{1/4}\sup_{0 \le s < \infty}\mu_n(s)\delta^2(\beta_0).
\]
\end{proof}
\section{Proofs of Results in Section~\ref{sec:NonConvex}}\label{appsec:NonConvex}
\subsection{Proof of Corollary~\ref{cor:NonLinearRegression}}
\begin{proof}
We will verify the assumptions of Theorem~\ref{thm:NewtonStepDifferentiable}. First note that
\begin{align*}
F_n(\theta) &= \frac{1}{n}\sum_{i=1}^n \left(Y_i - g(X_i^{\top}\theta)\right)^2,\\
\nabla F_n(\theta) &= -\frac{2}{n}\sum_{i=1}^n \left(Y_i - g(X_i^{\top}\theta)\right)g'(X_i^{\top}\theta)X_i,\\
\nabla_2F_n(\theta) &= \frac{2}{n}\sum_{i=1}^n \left\{g'(X_i^{\top}\theta)\right\}^2X_iX_i^{\top} - \frac{2}{n}\sum_{i=1}^n \left(Y_i - g(X_i^{\top}\theta)\right)g''(X_i^{\top}\theta)X_iX_i^{\top}.
\end{align*}
Thus for any $\theta\in\mathbb{R}^p$,
\begin{align*}
\nabla_2F_n(\theta) - \nabla_2F_n(\theta_0) &= \frac{2}{n}\sum_{i=1}^n \left\{\left(g'(X_i^{\top}\theta)\right)^2 - \left(g'(X_i^{\top}\theta_0)\right)^2\right\}X_iX_i^{\top}\\
&\quad-\frac{2}{n}\sum_{i=1}^n \left\{(Y_i - g(X_i^{\top}\theta))g''(X_i^{\top}\theta) - (Y_i - g(X_i^{\top}\theta_0))g''(X_i^{\top}\theta_0)\right\}X_iX_i^{\top}\\
&=: I - II.
\end{align*}
From assumption~\ref{eq:NonLinear}, we get that for any $1\le i\le n$,
\begin{align*}
\left|\left(g'(X_i^{\top}\theta)\right)^2 - \left(g'(X_i^{\top}\theta_0)\right)^2\right| &\le C_1^2(X_i)\norm{\theta - \theta_0}_2^2 + 2C_1(X_i)|g'(X_i^{\top}\theta_0)|\norm{\theta - \theta_0}_2,
\end{align*}
and
\begin{align*}
&\left|(Y_i - g(X_i^{\top}\theta))g''(X_i^{\top}\theta) - (Y_i - g(X_i^{\top}\theta_0))g''(X_i^{\top}\theta_0)\right|\\
&\,\le \left|(Y_i - g(X_i^{\top}\theta_0))\right|C_2(X_i)\norm{\theta - \theta_0}_2^{\alpha} + C_0(X_i)\left[|g''(X_i^{\top}\theta_0)|\norm{\theta - \theta_0}_2 + C_2(X_i)\norm{\theta - \theta_0}_2^{1 + \alpha}\right]
\end{align*}
Therefore,
\begin{align*}
&\norm{(\nabla_2F_n(\theta_0))^{-1}\left(\nabla_2F_n(\theta) - \nabla_2F_n(\theta_0)\right)}_{op}\\ &\quad\le \norm{\frac{2}{n}\sum_{i=1}^nC_1^2(X_i)\left(\nabla_2 F_n(\theta_0)\right)^{-1}X_iX_i^{\top}}_{op}\norm{\theta - \theta_0}_2^2\\
&\qquad+\norm{\frac{2}{n}\sum_{i=1}^n C_0(X_i)C_2(X_i)\left(\nabla_2 F_n(\theta_0)\right)^{-1}X_iX_i^{\top}}_{op}\norm{\theta - \theta_0}_2^{1 + \alpha}\\
&\qquad+\norm{\frac{2}{n}\sum_{i=1}^n \left\{2C_1(X_i)|g'(X_i^{\top}\theta_0)| + C_0(X_i)|g''(X_i^{\top}\theta_0)|\right\}\left(\nabla_2 F_n(\theta_0)\right)^{-1}X_iX_i^{\top}}_{op}\norm{\theta - \theta_0}_2\\
&\qquad+\norm{\frac{2}{n}\sum_{i=1}^n C_2(X_i)|Y_i - g(X_i^{\top}\theta_0)|\left(\nabla_2 F_n(\theta_0)\right)^{-1}X_iX_i^{\top}}_{op}\norm{\theta - \theta_0}_2^{\alpha}\\
&\le L_2(\theta_0)\norm{\theta - \theta_0}_2^{2} + L_{1 + \alpha}\norm{\theta - \theta_0}_2^{1 + \alpha} + L_1(\theta_0)\norm{\theta - \theta_0}_2 + L_{\alpha}(\theta_0)\norm{\theta - \theta_0}_2^{\alpha}.
\end{align*}
This completes the verification of condition~\eqref{eq:HolderContinuity} of Theorem~\ref{thm:NewtonStepDifferentiable} with right hand side there replaced by $\omega(\norm{\theta - \theta_0}_2)$, where for $r\ge 0,$
\[
\omega(r) = L_2(\theta_0)r^2 + L_{1 + \alpha}(\theta_0)r^{1 + \alpha} + L_1(\theta_0)r + L_{\alpha}(\theta_0)r^{\alpha}.
\]
Following the proof of Theorem~\ref{thm:NewtonStepDifferentiable}, the assumption~\eqref{eq:DerivativeSmall} implies the result.
\end{proof}
\section{Proofs of Results in Section~\ref{sec:Constrained}}\label{appsec:Constrained}
\subsection{Proof of Corollary~\ref{cor:EqualityConstrained}}
\begin{proof}
Define the function
\[
g_n(\beta, \nu) := \begin{bmatrix}\nabla F_n(\beta) + A^{\top}\nu\\A\beta - b\end{bmatrix}.
\]
It follows that
\[
\nabla g_n(\beta, \nu) := \begin{bmatrix}\nabla_2F_n(\beta) & A^{\top}\\ A &0\end{bmatrix}.
\]
So, $\beta^{\star}$ is a solution of the optimization problem if there exists a vector $\nu^{\star}$ such that $g_n(\beta^{\star}, \nu^{\star}) = 0$. From Theorem 2.2, it follows that if 
\begin{equation}\label{eq:CheckHolderContinuity}
\norm{[\nabla g_n(\beta_0, \nu_0)]^{-1}\left(\nabla g_n(\beta, \nu) - \nabla g_n(\beta_0, \nu_0)\right)}_{op} \le L\norm{\begin{pmatrix}\beta\\\nu\end{pmatrix} - \begin{pmatrix}\beta_0\\\nu_0\end{pmatrix}}_2^{\alpha},
\end{equation}
for $(\beta, \nu)$ in a ball around $(\beta_0, \nu_0)$ and if 
\begin{equation}\label{eq:CheckSmallGradient}
\norm{[\nabla g_n(\beta_0, \nu_0)]^{-1}g_n(\beta_0, \nu_0)}_2 \le (3L)^{-1/\alpha}.
\end{equation}
First note that
\[
\nabla g_n(\beta, \nu) - \nabla g_n(\beta_0, \nu_0) = \begin{bmatrix}\nabla_2F_n(\beta) - \nabla_2F_n(\beta_0)&0\\0&0\end{bmatrix},
\]
and using the inverse of a block matrix, we get $\left[\nabla g_n(\beta_0, \nu_0)\right]^{-1}\left(\nabla g_n(\beta, \nu) - \nabla g_n(\beta_0, \nu_0)\right)$ is given by
\[
\begin{bmatrix}\{I - [\nabla_2F_n(\beta_0)]^{-1}A^{\top}(A[\nabla_2F_n(\beta_0)]^{-1}A^{\top})^{-1}A\}[\nabla_2F_n(\beta_0)]^{-1}[\nabla_2F_n(\beta) - \nabla_2F_n(\beta_0)]&0\\0&0\end{bmatrix}.
\]
This implies that
\begin{align*}
&\norm{\left[\nabla g_n(\beta_0, \nu_0)\right]^{-1}\left(\nabla g_n(\beta, \nu) - \nabla g_n(\beta_0, \nu_0)\right)}_{op}\\ &\quad\le \norm{I - [\nabla_2F_n(\beta_0)]^{-1}A^{\top}(A[\nabla_2F_n(\beta_0)]^{-1}A^{\top})^{-1}A}_{op}\norm{[\nabla_2F_n(\beta_0)]^{-1}[\nabla_2F_n(\beta) - \nabla_2F_n(\beta_0)]}_{op}.
\end{align*}
Since
\begin{align*}
&\norm{[\nabla_2F_n(\beta_0)]^{-1}A^{\top}(A[\nabla_2F_n(\beta_0)]^{-1}A^{\top})^{-1}A}_{op}\\ &\quad= \norm{[\nabla_2F_n(\beta_0)]^{-1}A^{\top}(A[\nabla_2F_n(\beta_0)]^{-1}A^{\top})^{-1}A[\nabla_2F_n(\beta_0)]^{-1}A^{\top}}_{op} = 1,
\end{align*}
we get that
\[
\norm{\left[\nabla g_n(\beta_0, \nu_0)\right]^{-1}\left(\nabla g_n(\beta, \nu) - \nabla g_n(\beta_0, \nu_0)\right)}_{op} \le \norm{[\nabla_2F_n(\beta_0)]^{-1}[\nabla_2F_n(\beta) - \nabla_2F_n(\beta_0)]}_{op}.
\]
This proves the condition~\eqref{eq:CheckHolderContinuity}. For condition~\eqref{eq:CheckSmallGradient}, note that
\[
g_n(\beta_0,\nu_0) = \begin{bmatrix}\nabla F_n(\beta_0) + A^{\top}\nu_0\\0\end{bmatrix}.
\]
Again using the inverse of a block matrix, we get that $[\nabla g_n(\beta_0,\nu_0)]^{-1}g_n(\beta_0,\nu_0)$ is
\[
\begin{bmatrix}\{I - [\nabla_2F_n(\beta_0)]^{-1}A^{\top}(A[\nabla_2F_n(\beta_0)]^{-1}A^{\top})^{-1}A\}[\nabla_2F_n(\beta_0)]^{-1}(\nabla F_n(\beta_0) + A^{\top}\nu_0)\\\left(A[\nabla_2F_n(\beta_0)]^{-1}A^{\top}\right)^{-1}A[\nabla_2F_n(\beta_0)]^{-1}(\nabla F_n(\beta_0) + A^{\top}\nu_0)\end{bmatrix}.
\]
By the same reasoning, we have that
\begin{align*}
&\norm{[\nabla g_n(\beta_0,\nu_0)]^{-1}g_n(\beta_0,\nu_0)}_2\\ &\quad\le \left(1 + \norm{(A[\nabla_2F_n(\beta_0)]^{-1}A^{\top})^{-1}A}_{op}\right)\norm{[\nabla_2F_n(\beta_0)]^{-1}(\nabla F_n(\beta_0) + A^{\top}\nu_0)}_2.
\end{align*}
\end{proof}
\section{Proofs of Results in Section~\ref{sec:Applications}}\label{appsec:Applications}
\begin{proof}[Proof of Corollary~\ref{cor:LeaveOneOut}]
Theorem~\ref{thm:ConvexMEst} implies that
\begin{equation}\label{eq:FirstInequality}
\norm{\hat{\theta}_{-I,n} - \hat{\theta}_n + T_{-I,n}}_2 \le \max_{\substack{1\le i \le n,\\i\neq I}}\,\left\{C\left(1.5\norm{T_{-I,n}}_2, W_i\right) - 1\right\}1.5\norm{T_{-I,n}}_2,
\end{equation}
if
\begin{equation}\label{eq:FirstIneqCondition}
\max_{\substack{1\le i \le n,\,i\neq I}}\,C\left(1.5\norm{T_{-I,n}}_2, W_i\right) \le \frac{4}{3},
\end{equation}
where
\begin{equation}\label{eq:TwoTerms}
T_{-I,n} := \Bigg(\sum_{{1\le i\le n,\,i\neq I}}\nabla_2 L(\hat{\theta}_n, W_i)\Bigg)^{-1}\sum_{\substack{1\le i\le n,\,i\neq I}} \nabla L(\hat{\theta}_n, W_i).
\end{equation} 
To prove the result from this inequality, we need to simplify and control $T_{-I,n}$ and $\norm{T_{-I,n}}_2$.
Since $\hat{\theta}_n$ is the solution of the equation~\eqref{eq:ThetaHat}, we get for all $1\le I\le n$,
\[
\sum_{\substack{1\le i\le n,\,i\neq I}} \nabla L(\hat{\theta}_n, W_i) = -\nabla L(\hat{\theta}_n, W_I).
\]
Also, note that
\begin{align*}
&\norm{T_{-I,n} + n^{-1}\hat{\mathcal{Q}}_n^{-1}\nabla L(\hat{\theta}_n, W_I)}_2\\ &\qquad\le \norm{I - (n\hat{\mathcal{Q}}_n - \nabla_2L(\hat{\theta}_n, W_I))^{-1}n\hat{\mathcal{Q}}_n}_{op}\norm{n^{-1}\hat{\mathcal{Q}}_n^{-1}\nabla L(\hat{\theta}_n, W_I)}_2\\
&\qquad= \norm{I - (I - n^{-1}\hat{\mathcal{Q}}_n^{-1}\nabla_2L(\hat{\theta}_n, W_I))^{-1}}_{op}\norm{n^{-1}\hat{\mathcal{Q}}_n^{-1}\nabla L(\hat{\theta}_n, W_I)}_2\\
&\qquad= \norm{(I - n^{-1}\hat{\mathcal{Q}}_n^{-1}\nabla_2L(\hat{\theta}_n, W_I))^{-1}n^{-1}\hat{\mathcal{Q}}_n^{-1}\nabla_2L(\hat{\theta}_n, W_I)}_{op}\norm{n^{-1}\hat{\mathcal{Q}}_n^{-1}\nabla L(\hat{\theta}_n, W_I)}_2\\
&\qquad\le \frac{n^{-2}\norm{\hat{\mathcal{Q}}_n^{-1}\nabla_2L(\hat{\theta}_n, W_I)}_{op}}{1 - n^{-1}\norm{\hat{\mathcal{Q}}_n^{-1}\nabla_2L(\hat{\theta}_n, W_I)}_{op}}\norm{\hat{\mathcal{Q}}_n^{-1}\nabla L(\hat{\theta}_n, W_I)}_2\\ 
&\qquad= n^{-1}\norm{\hat{\mathcal{Q}}_n^{-1}\nabla_2L(\hat{\theta}_n, W_I)}_{op}\delta_{I,n}.
\end{align*}
From this inequality, it follows that $\norm{T_{-I,n}}_2 \le \delta_{I,n}.$ This inequality implies that condition~\eqref{eq:FirstIneqCondition} is satisfied if
\[
\max_{\substack{1\le i\le n,\,i\neq I}}\,C\left(1.5\delta_{I,n}, W_i\right) \le \frac{4}{3},
\]
which in turn implied by the condition~\eqref{eq:LOOCVCondition}. Substituting the inequalities above in~\eqref{eq:FirstInequality}, we get
\begin{align*}
&\norm{\hat{\theta}_{-I,n} - \hat{\theta}_n - n^{-1}\hat{\mathcal{Q}}_n^{-1}\nabla L(\hat{\theta}_n, W_I)}_2\\ &\qquad\le \frac{3}{2}\left[\max_{\substack{1\le i\le n,\,i\neq I}}\,C(1.5\delta_{I,n}, W_i) - 1 + n^{-1}\norm{\hat{\mathcal{Q}}_n^{-1}\nabla_2L(\hat{\theta}_n, W_I)}_{op}\right]\delta_{I,n}.
\end{align*}
\end{proof}
\begin{proof}[Proof of Corollary~\ref{cor:LeaveMoreOut}]
Theorem~\ref{thm:ConvexMEst} implies that
\begin{equation}\label{eq:FirstInequalityMore}
\norm{\hat{\theta}_{-I,n} - \hat{\theta}_n + T_{-I,n}}_2 \le \max_{\substack{1\le i \le n,\,i\notin I}}\,\left\{C\left(1.5\norm{T_{-I,n}}_2, W_i\right) - 1\right\}1.5\norm{T_{-I,n}}_2,
\end{equation}
if
\begin{equation}\label{eq:FirstIneqConditionMore}
\max_{\substack{1\le i \le n,\,i\neq I}}\,C\left(1.5\norm{T_{-I,n}}_2, W_i\right) \le \frac{4}{3},
\end{equation}
where
\begin{equation}\label{eq:TwoTermsMore}
T_{-I,n} := \Bigg(\sum_{{1\le i\le n,\,i\notin I}}\nabla_2 L(\hat{\theta}_n, W_i)\Bigg)^{-1}\sum_{\substack{1\le i\le n,\,i\notin I}} \nabla L(\hat{\theta}_n, W_i).
\end{equation} 
To prove the result from this inequality, we need to simplify and control $T_{-I,n}$ and $\norm{T_{-I,n}}_2$.
Since $\hat{\theta}_n$ is the solution of the equation~\eqref{eq:ThetaHat}, we get for all $1\le I\le n$,
\[
\sum_{\substack{1\le i\le n,\,i\notin I}} \nabla L(\hat{\theta}_n, W_i) = -\sum_{i\in I}\nabla L(\hat{\theta}_n, W_I).
\]
Also, note that
\begin{align*}
&\norm{T_{-I,n} + n^{-1}\hat{\mathcal{Q}}_n^{-1}\sum_{i\in I}\nabla L(\hat{\theta}_n, W_I)}_2\\ &\quad\le \norm{I - \left(n\hat{\mathcal{Q}}_n - \sum_{i\in I}\nabla_2L(\hat{\theta}_n, W_i)\right)^{-1}n\hat{\mathcal{Q}}_n}_{op}\norm{n^{-1}\hat{\mathcal{Q}}_n^{-1}\sum_{i\in I}\nabla L(\hat{\theta}_n, W_I)}_2\\
&\quad= \norm{I - \left(I - n^{-1}\hat{\mathcal{Q}}_n^{-1}\sum_{i\in I}\nabla_2L(\hat{\theta}_n, W_i)\right)^{-1}}_{op}\norm{n^{-1}\hat{\mathcal{Q}}_n^{-1}\sum_{i\in I}\nabla L(\hat{\theta}_n, W_I)}_2\\
&\quad= \norm{\left(I - n^{-1}\hat{\mathcal{Q}}_n^{-1}\sum_{i\in I}\nabla_2L(\hat{\theta}_n, W_i)\right)^{-1}n^{-1}\hat{\mathcal{Q}}_n^{-1}\sum_{i\in I}\nabla_2L(\hat{\theta}_n, W_i)}_{op}\norm{n^{-1}\hat{\mathcal{Q}}_n^{-1}\sum_{i\in I}\nabla L(\hat{\theta}_n, W_I)}_2\\
&\quad\le \frac{n^{-1}\norm{\hat{\mathcal{Q}}_n^{-1}\sum_{i\in I}\nabla_2L(\hat{\theta}_n, W_i)}_{op}}{1 - n^{-1}\norm{\hat{\mathcal{Q}}_n^{-1}\sum_{i\in I}\nabla_2L(\hat{\theta}_n, W_i)}_{op}}\norm{n^{-1}\hat{\mathcal{Q}}_n^{-1}\sum_{i\in I}\nabla L(\hat{\theta}_n, W_I)}_2\\ 
&\quad= n^{-1}\norm{\hat{\mathcal{Q}}_n^{-1}\sum_{i\in I}\nabla_2L(\hat{\theta}_n, W_i)}_{op}\delta_{I,n}.
\end{align*}
From this inequality, it follows that $\norm{T_{-I,n}}_2 \le \delta_{I,n}.$ This inequality implies that condition~\eqref{eq:FirstIneqConditionMore} is satisfied if
\[
\max_{\substack{1\le i\le n,\,i\notin I}}\,C\left(1.5\delta_{I,n}, W_i\right) \le \frac{4}{3},
\]
which in turn implied by the condition~\eqref{eq:LOOCVCondition}. Substituting the inequalities above in~\eqref{eq:FirstInequalityMore}, we get
\begin{align*}
&\norm{\hat{\theta}_{-I,n} - \hat{\theta}_n - n^{-1}\hat{\mathcal{Q}}_n^{-1}\sum_{i\in I}\nabla L(\hat{\theta}_n, W_I)}_2\\ &\qquad\le \frac{3}{2}\left[\max_{\substack{1\le i\le n,\,i\notin I}}\,C(1.5\delta_{I,n}, W_i) - 1 + n^{-1}\norm{\hat{\mathcal{Q}}_n^{-1}\sum_{i\in I}\nabla_2L(\hat{\theta}_n, W_i)}_{op}\right]\delta_{I,n}.
\end{align*}
\end{proof}

\end{document}